\documentclass[a4paper,12pt]{article}
\usepackage[latin1]{inputenc}
\usepackage{amsfonts}
\usepackage{amsthm}
\usepackage{amsmath}
\usepackage{amsfonts}
\usepackage{latexsym}
\usepackage{amssymb}
\usepackage{dsfont}

\newtheorem{fed}{Definition}[section]
\newtheorem{teo}[fed]{Theorem}
\newtheorem*{teo*}{Theorem}
\newtheorem{lem}[fed]{Lemma}
\newtheorem{cor}[fed]{Corollary}
\newtheorem{pro}[fed]{Proposition}
\theoremstyle{definition}
\newtheorem{rem}[fed]{Remark}
\newtheorem{nota}[fed]{Notations}

\newtheorem{exa}[fed]{Example}

\def\coma{\, , \, }

\def\py{\peso{and}}
\newcommand{\peso}[1]{ \quad \text{ #1 } \quad }

\def\n0{n_{ \text{\rm \tiny o}}}

\newcommand{\IN}[1]{\mathbb {I} _{#1}}
\def\In{\mathbb {I} _n}

\def\suml{\sum\limits}
\def\prodl{\prod\limits}
\def\QEDP{\tag*{\QED}}

\def\bce{\begin{center}}
\def\ece{\end{center}}

\DeclareMathOperator{\FP}{FP\,}

\def\cD{\mathcal D}

\def\py{\peso{and}}
\def\rk{\text{\rm rk}}
\def\noi{\noindent}
\def\cF{\mathcal F}
\def\cG{\mathcal G}
\def\QED{\hfill $\square$}
\def\EOE{\hfill $\triangle$}
\def\EOEP{\tag*{\EOE}}
\def\uno{\mathds{1}}
\def\bm{\left[\begin{array}}
\def\em{\end{array}\right]}
\def\ben{\begin{enumerate}}
\def\een{\end{enumerate}}
\def\bit{\begin{itemize}}
\def\eit{\end{itemize}}
\def\barr{\begin{array}}
\def\earr{\end{array}}
\def\igdef{\ \stackrel{\mbox{\tiny{def}}}{=}\ }

\def\la{\lambda}
\def\al{\alpha}
\def\N{\mathbb{N}}
\def\M{\mathbb{M}}

\def\R{\mathbb{R}}
\def\C{\mathbb{C}}

\def\cA{\mathcal{A}}

\def\F{\mathcal{F}}
\def\G{\mathcal{G}}
\def\cH{\mathcal{H}}
\def\cK{\mathcal{K}}

\def\cS{{\cal S}}

\def\cN{{\cal N}}
\def\cV{{\cal V}}

\def\cW{{\cal W}}

\def\rai{^{1/2}}

\def\cX{\mathcal{X}}

\def\cZ{\mathcal{Z}}

\def\orto{^\perp}
\def\inc{\subseteq}

\def\inv{^{-1}}
\def\rai{^{1/2}}

\def\api{\langle}
\def\cpi{\rangle}

\def\ua{^\uparrow}
\def\da{^\downarrow}

 \DeclareMathOperator{\tr}{tr}

\DeclareMathOperator{\leqp}{\leqslant}
\DeclareMathOperator{\geqp}{\geqslant}
\DeclareMathOperator{\convf}{Conv (\R_{\ge0})}
\DeclareMathOperator{\convfs}{Conv_s (\R_{\ge0})}

\newcommand{\hil}{\mathcal{H}}
\newcommand{\op}{B(\mathcal{H})}

\newcommand{\posop}{B(\mathcal{H})^+}

\newcommand{\mat}{\mathcal{M}_d(\mathbb{C})}

\newcommand{\matpos}{\mat^+}

\newcommand{\matinvd}{\mathcal{G}\textit{l}\,(d)}

\def\beq{\begin{equation}}
\def\eeq{\end{equation}}

\def\pausa{\medskip\noi}

\begin{document}
\title{Aliasing and oblique dual pair designs \\for consistent sampling}
%\author{Mar\'\i a J. Benac, Pedro G. Massey, and Demetrio Stojanoff
%\thanks{Partially supported by CONICET 
%(PIP 0435/10) and  Universidad Nacional de La PLata (UNLP 11 X585).} }
\author{Mar\'\i a J. Benac, Pedro G. Massey, and Demetrio Stojanoff 
\footnote{Partially supported by CONICET 
(PIP 0435/10) and  Universidad Nacional de La PLata (UNLP 11X681) %\hspace{19cm}
 e-mail addresses: mjbenac@gmail.com , massey@mate.unlp.edu.ar , demetrio@mate.unlp.edu.ar}
\\ 
{\small Depto. de Matem\'atica, FCE-UNLP,  La Plata, Argentina
and IAM-CONICET  }}
%\author{M.J. Benac, P. Massey and D. Stojanoff}
\date{}
\maketitle

\begin{abstract}
In this paper we study some aspects of oblique duality between finite sequences of vectors $\cF$ and $\cG$ lying in 
finite dimensional subspaces $\cW$ and $\cV$, respectively. We compute the possible eigenvalue lists of the frame operators of oblique duals to $\cF$ lying in $\cV$; we then compute the spectral and geometrical structure of
minimizers of convex potentials among oblique duals for $\cF$ under some restrictions. We obtain a complete quantitative 
analysis of the impact that the relative geometry between the subspaces $\cV$ and $\cW$ has in oblique duality.
We apply this analysis to compute those rigid rotations $U$ for $\cW$ such that the canonical oblique dual of $U\cdot \cF$ 
minimize every convex potential; we also introduce a notion of aliasing for oblique dual pairs and compute  
those rigid rotations $U$ for $\cW$ such that the canonical oblique dual pair associated to $U\cdot \cF$ minimize the aliasing. We point out that these two last problems are intrinsic to the theory of oblique duality.
\end{abstract}


\tableofcontents

\section{Introduction}

A finite sequence $\cF=\{f_i\}_{i\in\In}$ is a frame for a Hilbert space $\cW\cong\C^d$ if $\cF$ spans $\cW$, where $\In=\{1,\ldots,n\}$.  
In this case, a sequence $\cG=\{g_i\}_{i\in\In}$ in $\cW$ is a (classical) dual for $\cF$ in $\cW$ if the following reconstruction formulas holds:
\begin{equation}\label{int 1}
 f=\sum_{i\in\In} \langle f\coma  f_i\rangle \ g_i= \sum_{i\in\In} \langle f\coma  g_i\rangle \ f_i \ , \ \ f\in\cW\ .
 \end{equation}
Hence, frames allow for linear encoding-decoding schemes of vectors in $\cW$ in terms of linear generators for $\cW$. 
Moreover, in case $n>d$ then the set of dual frames for $\cF$ in $\cW$ has a rich structure which plays a key role in applications of finite frame theory to real life situations, such as signal transmission through noisy channels (see \cite{FF,Chrisbook}). Similarly, applications of finite frame theory have lead to consider the so-called frame design problems, i.e. the existence and construction of frames with prescribed properties, based on the flexibility of finite frames  
(see \cite{BF,dhan,MR12,Pot}). 

An important aspect of frames is that of its numerical stability; typically, numerical stability is measured in terms of 
the spread of the eigenvalues of the so-called frame operator $S_\cF$ of a frame $\cF=\{f_i\}_{i\in\In}$, that is given by $S_{\cF}=\sum_{i\in\In}f_i\otimes f_i$. One of the most important measures of the spread of the spectrum of $S_\cF$ is given by the frame potential of $\cF$ (see \cite{BF}) given by $\FP(\cF)=\sum_{i,\,j\in\In} |\langle f_i,\,f_j\rangle |^2=\tr(S_\cF^2)$. Indeed, it turns out that minimizers of the frame potential - within appropriate sets of frames - minimize the spread of the spectrum of their frame operators (see \cite{BF,CKFT}). Recently, there has also been interest in the structure of minimizers of the so-called mean squared error of a frame $\cF$ given by MSE$(\cF)=\tr(S_\cF^{-1})$ - within convenient sets of frames (see \cite{FMP}). This raises the question of whether the minimizers of these two different functionals coincide.
It turns out that there is a natural and structural measure of spread of the spectrum of the frame operators, called submajorization, that has proved useful in explaining the spectral and geometrical structure of both frame potential and mean squared error minimizers (see \cite{MR10,MRS13,MRS14}).

In the seminal paper \cite{YEldar1} Y. Eldar 
developed the theory of oblique duality for finite frames, which is an extended setting for 
linear encoding-decoding schemes in a Hilbert space $\cW$, based on the notion of consistent sampling. As the starting point for this theory, we consider $\cW,\,\cV\subset \hil$ two subspaces of a finite dimensional Hilbert space $\hil$, such that  $\cW \oplus\cV\orto=\hil$ (i.e.  $\cW^\perp + \cV=\hil$ and $\cW^\perp \cap \cV=\{0\}$). Given a frame $\cF=\{f_i\}_{i\in\In}$ for $\cW$ and a frame $\cG=\{g_i\}_{i\in\In}$ for $\cV$ we say that $\cG$ is an oblique dual of $\cF$ if the following reconstruction formula holds 
$$ 
f=\sum_{i\in\In} \langle f\coma g_i\rangle \, f_i \peso{for every} f\in\cW\ .
$$ 
The theory of oblique duality has been both developed and extended in several ways (see \cite{AnAnCo,AnCo,DvEld,CE06,YEldar2}).
On the other hand, it has been successfully applied to study duality for finitely generated shift invariant 
systems on $L^2(\R^d)$ (see \cite{YEldar3,HG07,HG09}). 

There are, however, some aspects of oblique duality that remain to be explored, even in the finite dimensional case.
In this paper, based on several tools coming from matrix analysis, we consider the following problems in oblique duality.
On the one hand,  we study the spectral (and geometrical) structure of oblique duals of the frame $\cF$ for $\cW$ that lie in $\cV$. In this case, we obtain an explicit description of the eigenvalues of the frame operators of oblique duals. With this description at hand, we compute the structure of minimizers of submajorization within the set of oblique duals of $\cF$ under some restrictions. These optimal oblique duals for submajorization turn out to minimize the so-called convex potentials (that include both the frame potential and the mean squared error). 

On the other hand, it has been noticed that the relative position of the subspaces $\cV$ and $\cW$ for which $\cW\oplus \cV^\perp=\hil$ plays a key role when comparing oblique duality to classical duality. This phenomenon has been explored mainly though the angle between the subspaces $\cV$ and $\cW$.
Yet, the angle between the subspaces $\cV$ and $\cW$ only provides qualitative measure of the role of the relative geometry of $\cV$ and $\cW$ in the context of $\cV$-duality. In this paper, we give a detailed description of the role of the relative position of $\cV$ and $\cW$ in the oblique duality of $\cF$ in case the subspaces are finite dimensional. Our analysis relies on a multiplicative Lidskii's inequality and it is based on the complete list of the so-called principal angles between $\cV$ and $\cW$. Our results provide sharp quantitative measures of these relations.

We also consider two problems that are intrinsic to oblique duality. We first notice that the so-called canonical oblique dual $(U\cdot \cF)^\#_\cV$ of a rigid rotation $U\cdot \cF=\{U\,f_i\}_{i\in\In}$ of $\cF$ is, in general, not a rigid rotation of the canonical oblique dual $\cF^\#_\cV$ of $\cF$ (as opposed to classical duality). 
Hence, we compute the rigid rotation $U_0$ of $\cW$ such that the oblique dual $(U_0\cdot\cF)^\#_\cV$ is optimal with respect to submajorization, among all such rigid rotations. Again, this implies a family of inequalities in terms of convex potentials that are relevant for numerical analysis purposes. We also compute the exact value of the aliasing norm of the consistent sampling corresponding to subspaces $\cV$ and $\cW$ and introduce a notion of aliasing for oblique dual pairs. In this context we compute the optimal rigid rotations $U_0$ that minimize the aliasing of the dual pairs $(U\cdot \cF,(U\cdot \cF)^\#_\cV)$ for the fixed frame $\cF$. 

Throughout the paper we consider finite sequences of vectors $\cF$ that are frames 
for finite dimensional subspaces $\cW$ of a possibly infinite dimensional Hilbert space $\hil$, since  
this is the setting that we shall need for future applications of the results herein; on the other hand, 
the assumption that $\hil$ is finite dimensional does not provide any substantial simplification in the proofs of our results.

The paper is organized as follows. In Section \ref{sec prelims} we describe the basic framework of 
oblique duality between finite sequences of vectors, together with some basic facts about convex potentials. In order to deal with these general convex potentials, we also consider submajorization and log-majorization, which are spectral relations between positive finite rank operators (or positive matrices). In particular, we include a multiplicative analogue of Lidskii's additive inequality that plays a crucial role in this note. In Section \ref{Sec 3} we obtain a convenient parametrization of the set of oblique duals of a fixed frame and use it to compute the possible eigenvalues of the frame operators of oblique duals. We then compute the structure of optimal oblique duals for submajorization, under certain restrictions. In Section \ref{sec optimal rot1}, after recalling some standard notions from functional analysis, we compute the rigid rotations $U$ of $\cW$ such that the spread of the eigenvalues of the frame operator of the oblique dual $(U\cdot\cF)^\#_\cV$ is minimal, with respect to submajorization. We also consider the combination of the problems of Section \ref{Sec 3} and Section \ref{sec optimal rot1} i.e., the properties of the optimal oblique dual frame (with norm restrictions) corresponding to the optimal rotation of $\cF$.
In Section \ref{sec 5} we compute the exact value of the aliasing norm of the consistent sampling based on the subspaces 
$\cV$ and $\cW$ and introduce a notion of aliasing for arbitrary oblique dual pairs $(\cF,\cG)$. In this context, we compute the rigid rotations $U$ of $\cW$ that minimize the aliasing for the oblique dual pair $(U\cdot \cF,(U\cdot \cF)^\#_\cV)$.

\section{Preliminaries}\label{sec prelims}

In this section we introduce the notations and basic terminology of frame theory and oblique duality. We also describe 
the convex potentials for finite sequences of vectors (that contain, for example, the Benedetto-Fickus' frame potential) which will serve as a numerical measures of the (relative) spread of the eigenvalues of the frame operators. Finally, we describe some notions from matrix theory that will allow us to deal with these general convex potentials.

\subsection{Oblique dual frames and convex potentials}\label{sec defi oblique duals}

In what follows we consider a fixed complex separable Hilbert space $\hil$. We take 
\beq\label{barraM}
\M=\IN{p} \igdef \{1\coma \ldots \coma p\}  \peso{for \ \ $p\in \N$ \quad or} \M=\N 
\eeq
in such a way that $\dim \hil=|\M|$. 
Let $\cW$ be a closed subspace of $\hil$. Recall that a sequence $\cF=\{f_i\}_{i\in I}$ in $\cW$ is a {\it frame} for $\cW$ if there exist $0<A<B$ such that
\beq\label{defi W frame} 
A \, \|f\|^2\leq \sum_{i\in I}|\langle f\coma f_i\rangle |^2\leq B\,\|f\|^2 \ , \ \ f\in \cW\, . 
\eeq If only the inequality to the right holds, we say that $\cF$ is a Bessel sequence.

\pausa 
In general, given a Bessel sequence $\cF=\{f_i\}_{i\in I}$ we consider its {\it synthesis operator} $T_\cF\in B(\ell^2(I),\hil)$ given by 
$T_\cF((a_i)_{i\in I})=\sum_{i\in I} a_i\ f_i$ which, by hypothesis on $\cF$, is a bounded linear transformation. We also consider $T_\cF^*\in B(\hil,\ell^2(I))$ called the {\it analysis operator} of $\cF$, given by $T_\cF^*(f)=(\langle f,f_i\rangle )_{i\in I}$ and the {\it frame operator} of $\cF$ defined by $S_\cF=T_\cF\,T_\cF^*$. 
It is straightforward to check that
$$ 
\langle S_\cF f\coma f \rangle =\sum_{i\in I}|\langle f\coma f_i\rangle |^2 \ , \ \ f\in \hil\ . 
$$
Hence, $S_\cF$ is a positive semidefinite operator; moreover, a Bessel sequence $\cF$ in $\cW$ 
is a frame for $\cW$ if and only if $S_\cF$ is an invertible operator when restricted to $\cW$ or, equivalently if $T_\cF$ is a surjective operator onto $\cW$.

\pausa
In order to describe oblique duality, we fix two closed subspaces $\cV,\,\cW\inc\hil$ such that $\cW^\perp \oplus \cV=\hil$, that is such that $\cW^\perp + \cV=\hil$ and $\cW^\perp \cap \cV=\{0\}$. Hence, $\cW^\perp$ is a common (algebraic) complement of $\cW$ and $\cV$.
It is well known that in this case $P_\cW|_\cV:\cV\rightarrow \cW$ is a linear bounded isomorphism so, in particular,
we see that $\dim \cV=\dim \cW$ as Hilbert spaces. Moreover, the conditions $\cW^\perp \oplus \cV=\hil$ and $\cW \oplus \cV^\perp=\hil$ are actually equivalent.

\pausa
Fix a frame $\cF=\{f_i\}_{i\in I}$ for $\cW$. Following \cite{YEldar1,YEldar2} (see also  \cite{CE06}), given a Bessel sequence $\cG=\{g_i\}_{i\in I}$  in $\cV$ we say that $\cG$ is a  (oblique) $\cV$-dual of $\cF$   if 
$$ 
g=\sum_{i\in I} \langle g\coma f_i\rangle \ g_i = T_\cG\,T_\cF^* \, g \peso{for every} g\in \cV\, .
$$ 
It turns out (see \cite{YEldar1,YEldar2}) that $\cG$ is a $\cV$-dual of $\cF$ if and only if
$T_\cG\,T_\cF^*=P_{\cV//\cW^\perp}$, where $P_{\cV//\cW^\perp}$ denotes the oblique projection with range $\cV$ and null space $\cW^\perp$.
Hence, $T_\cG$ is surjective onto $\cV$ and then $\cG$ is a frame for $\cV$; by taking adjoints in the identity $T_\cG\,T_\cF^*=P_{\cV//\cW}$ we also get that $T_\cF\,T_\cG^*= P_{\cV//\cW^\perp}^*=P_{\cW//\cV^\perp}$ i.e.
$$ 
f=\sum_{i\in I} \langle f \coma g_i\rangle \ f_i \peso{for every} f\in \cW\, .
$$
We shall consider the set of oblique $\cV$-duals of $\cF$ given by
\begin{equation}\label{eq duales}
\cD_\cV(\cF) \igdef \left\{\cG  
\in\cV^{\, I} : \cG \text{ is a }\cV\text{-dual of } \cF\right\}\ .
\end{equation}

\begin{rem}\label{rem classical canonical dual}
Let $\cF=\{f_i\}_{i\in I} $ be a frame for $\cW$. If we set $\cV=\cW$ then a Bessel sequence $\cG$ in $\cW$ is a $\cW$-dual of $\cF$ if it is a dual frame for $\cF$ in the classical sense (see \cite{Chrisbook}) i.e. $T_\cG\,T_\cF^*=P_\cW$. Hence
$$
\cD_\cW(\cF)=\cD(\cF)\igdef 
\left\{\cG 
\in\cW^{\, I} : \cG\text{ is a dual frame for } \cF \text{ in } \cW\,   \right\}\ .
$$ Recall that there is a distinguished (classical) dual, called the canonical dual of $\cF$, denoted 
$$
\cF^\#=\{f_i^\#\}_{i\in I}  \peso{given by} f_i^\#=S_\cF^\dagger \,f_i  \peso{for every}  i\in I  \ , 
$$
where $S_\cF^\dagger$ denotes the Moore-Penrose pseudo-inverse of the (closed range 
operator) $S_\cF\,$.
\EOE
\end{rem}

\pausa
In the general setting for oblique duality there also exists a distinguished $\cV$-dual for $\cF$, the so-
called {\it canonical $\cV$-dual} which we denote by 
\beq\label{canVdual}
\cF^\#_\cV=\{f^\#_{\cV,\,i}\}_{i\in I} \peso{given by} 
f^\#_{\cV,\,i} =P_{\cV//\cW^\perp} \, f^\#_i  = P_{\cV//\cW^\perp}\, S_\cF^\dagger \, f_i \peso{for every}  i\in I \ ,
\eeq
where $\cF^\#=\{f_i^\#\}_{i\in I}\in\cD(\cF)$ denotes the (classical) canonical dual as described in Remark \ref{rem classical canonical dual}. It turns out that the encoding-decoding scheme based on the oblique dual pair $(\cF\coma \cF^\#_{\cV})$ has several optimality properties (see \cite{YEldar1,YEldar2}).

\medskip

\noindent {\bf Convex potentials for finite sequences in $\hil$}

\pausa 
In their seminal work \cite{BF}, Benedetto and Fickus introduced a functional defined on finite sequences of (unit norm) vectors, the so-called {\it frame potential}, given by 
\beq\label{rep BFpot}
\FP(\{f_i\}_{i\in \In} )
=\sum_{i,\,j\,\in \In}|\api f_i\coma f_j \cpi |\,^2\ .
\eeq 
In case $\dim \hil=p\in\N$ then one of their major results shows that tight unit norm frames - which form an important class of frames because of their simple reconstruction formulas and robustness properties - can be characterized as (local) minimizers of this functional among unit norm frames. Since then, there has been interest in (local) minimizers of the frame potential within certain classes of frames, since such minimizers can be considered as natural substitutes of tight frames (see for example \cite{CKFT,MR10}). 
Notice that, given $\cF=\{f_i\}_{i\in \IN{n}}\in  \cH^n$ then $\FP(\cF)=\tr \, S_\cF^2$. Recently, there has been interest in the structure of frames that minimize other potentials such as the so-called mean squared error (MSE) given by $\text{MSE}(\cF)=\tr(S_\cF^{-1})$ (see \cite{FMP,MRS13,Pot}). Next, we describe a broad family of potentials introduced in \cite{MR10}, that contain both the frame potential and the MSE.

\pausa
In what follows we consider the sets
$$
\convf = \{ 
h:[0 \coma \infty)\rightarrow [0 \coma \infty): h   \ \mbox{ is a convex function }  \ \} 
$$  
and $\convfs = \{h\in \convf : h$ is strictly convex $\}$.

\begin{fed}\label{pot generales}\rm
Given $h\in \convf$ then the {\it convex potential} associated to $h$, denoted by $P_h$, is defined as follows:
for a finite sequence $\cF=\{f_i\}_{i\in \IN{n}}\in  \cH^n$ with  $\cW=\text{Span}\{f_i: \ i\in\In\}$ and frame operator $S_\cF\in B(\hil)^+$, then
$$
%\barr{rl}
P_h(\cF)=\sum_{i\in\IN{d}}h(\lambda_i((S_\cF)_\cW))\ ,
%P_h(\cF)&=\tr \, h((S_\cF)_\cW) 
%\peso {for}  \cF=\{f_i\}_{i\in \IN{n}}\in  \cH^n \ ,
%\ \ \cW=\text{Span}\{f_i: \ i\in\In\}
% \earr
$$ where $d=\dim \cW$ and $(\lambda_i((S_\cF)_\cW))_{i\in\IN{d}}\in \R_{>0}^d$ denotes the vector of eigenvalues of the compression $(S_\cF)_\cW\in B(\cW)^+$, counting multiplicities and arranged in non-increasing order. 
%: hence $\tr \, h((S_\cF)_\cW)$ is well defined). 
\EOE
%$\tr \, h((S_\cF)_\cW) $
%where $h((S_\cF)_\cW)\in B(\cW)^+$ is obtained by means of the usual functional calculus from the compression $(S_\cF)_\cW$ 
\end{fed}

\begin{rem}\label{algo sobre compresion}
With the notations of Definition \ref{pot generales}, 
notice that by construction  $\cW$ is a reductive subspace for $S_\cF$, and hence $(S_\cF)_\cW$ is a well defined positive operator acting on the finite dimensional subspace $\cW$. Moreover, $$P_h(\cF)= \tr \, h((S_\cF)_\cW)$$
where $h((S_\cF)_\cW)\in B(\cW)^+$ is obtained by means of the usual functional calculus from the compression $(S_\cF)_\cW$
and the trace is taken in the finite dimensional Hilbert space $\cW$.
Therefore, in case $h\in\convf$ is such that $h(0)=0$ we get that $$P_h(\cF)=\tr \, h((S_\cF)_\cW)=\tr \, h(S_\cF)\ .$$ In particular, we see that if $h(x)=x^2$ then $P_h(\cF)$ coincides with the frame potential. \EOE
\end{rem}

\pausa
Fix $h\in \convf$ and consider its associated convex potential $P_h$. If $\cF=\{f_i\}_{i\in \In}$ is a finite sequence in $\cH^n$ then $P_h(\cF)$ is a measure of the spread of the eigenvalues of the frame operator of $\cF$. That is, (under suitable normalization hypothesis on $\cF$) the smaller the value $P_h(\cF)$ is, the more concentrated the non-zero eigenvalues of $S_\cF$ are (see \cite{MR10,MRS13,MRS14,Pot}).

\pausa
In order to deal with these general convex potentials we consider the notions of submajorization and log-majorization in the next section.

\subsection{(Log-)majorization and convex functions}\label{sec 2.2.}

Next we briefly describe sub-majorization, majorization and log-majorization, that are notions from matrix analysis. For a detailed exposition on these relations see \cite{Bat}.

\pausa
 Given $x,\,y\in \R_{\geq 0}^d$ we say that $x$ is
{\it submajorized} by $y$, and write $x\prec_w y$,  if
$$
\suml_{i=1}^k x^\downarrow _i\leq \suml_{i=1}^k y^\downarrow _i \peso{for every} k\in \mathbb I_d=\{1,\ldots,d\} \ ,
$$
where $z\da\in\R^d$ (respectively $z\ua\in \R^d$) denotes the vector obtained by re-arrangement of the entries of $z\in\R^d$ in non-increasing (respectively non-decreasing) order.
If $x\prec_w y$ and $\tr x = \sum_{i=1}^dx_i=\sum_{i=1}^d y_i = \tr y$,  then $x$ is
{\it majorized} by $y$, and write $x\prec y$.

\pausa
Log-majorization between vectors in $\R_{\geq 0}^d$ is a multiplicative analogue of majorization in $\R_{\geq 0}^d$. Indeed, given $x,\, y\in \R_{\geq 0}^d$ we say that $x$ is {\it log-majorized} by $y$, denoted $x\prec_{\log} y$, if
$$
\prodl_{i=1}^k x^\downarrow _i\leq \prodl_{i=1}^k y^\downarrow _i \peso{for every} k\in \mathbb I_{d-1} \peso{and}
\prodl_{i=1}^d x^\downarrow _i= \prodl_{i=1}^d y^\downarrow _i \ .$$
It is known (see \cite{Bat}) that if $x,\, y\in \R_{\geq  0}^d$ are such that $x\prec_{\log}$ then $x\prec_w y\ $.
On the other hand we write
$x \leqp y$ if $x_i \le y_i$ for every $i\in \mathbb I_d \,$.  It is a standard  exercise
to show that if $x,\, y\in \R_{\geq 0}^d$ then $x\leqp y \implies x^\downarrow\leqp y^\downarrow \implies x\prec_{\log} y \implies x\prec_w y $.

\pausa
Our interest in majorization is motivated by the
relation of this notion with tracial inequalities
for convex functions. Indeed, 
 given $x,\,y\in \R_{\geq 0}^d$ and  $h\in\convf $, then (see for example \cite{Bat}):
\ben
\item If one assumes that $x\prec y$, then
$
\tr h(x) \igdef\suml_{i=1}^dh(x_i)\leq \suml_{i=1}^dh(y_i)=\tr h(y)\ .
$
\item If only $x\prec_w y$,  but the map $h$ is also non-decreasing, then  still
$\tr h(x) \le \tr h(y)$.
\item If $x\prec_w y$, $h\in\convfs$ is non-decreasing and $\tr \,h(x) =\tr \, h(y)$, then there exists a permutation $\sigma$
of $\IN{d}$ such that $y_i=x_{\sigma(i)}$ for $i\in \IN{d}\,$.
\een

\pausa
The following result is a multiplicative Lidskii's inequality for matrices, that also contains a detailed description of the case of {\it equality}. In what follows, given $x=(x_i)_{i\in\IN{d}},\,y=(y_i)_{i\in\IN{d}}\in\R^d$ then $x\circ y=(x_i\,y_i)_{i\in\IN{d}}\in\R^d$ denotes the entry-wise product of the vectors. Also, given a selfadjoint matrix $A\in\mat$
then $\la(A)\in\R^d$ denotes the eigenvalues of $A$, counting multiplicities and arranged in non-increasing order.

\begin{teo}[\cite{MRS14}]\label{teo hay max y min mayo}
Let $S\in\matinvd^+$ and let $\lambda\in (\R_{>0}^d)\da$.
Then, for every $V\in\mat$  such that $\lambda(V^*V)=\lambda$ we have that
\beq\label{max y min logamyo}
\lambda(S)\circ \lambda\ua\prec_{\log}\lambda(VSV^*)\prec_{\log}\lambda(S)\circ \lambda
\in (\R_{>0}^d)\da \ .
\eeq
Moreover, if $\lambda(VSV^*)=(\lambda(S)\circ \lambda\ua)\da$
(resp.
$\lambda(VSV^*)=\lambda(S)\circ \lambda$) then there exists an o.n.b.
$\{ v_i\}_{i\in \IN{d}}$ of $\C^d$ such that
\beq\label{hay base}
S=\sum_{i\in \IN{d}} \lambda_i(S)\ v_i\otimes v_i \peso{and} |V|=
\sum_{i\in \IN{d}} \lambda_{d+1-i}\rai\ v_i\otimes v_i\
\eeq
\big(\,resp. $S=\sum_{i\in \IN{d}} \lambda_i(S)\ v_i\otimes v_i$ and
$|V|=\sum_{i\in \IN{d}} \lambda_{i}\rai\ v_i\otimes v_i$\,\big).\qed
\end{teo}

\section{Spectral structure and optimal oblique duals}\label{Sec 3}

In this section we obtain a simple and explicit description of the eigenvalues of the frame operators of oblique duals of a fixed frame. We then apply this result to study the existence and structure of oblique duals that are optimal among oblique duals with some restrictions. 

\subsection{Spectral structure of oblique duals}

Let $\cV$ and $\cW$ be closed subspaces of the Hilbert space $\cH$ such that $\cW^\perp\oplus\cV=\cH$. Let $\cF=\{f_i\}_{i\in I}$ be a sequence in $\cW$ that is a frame for $\cW$. There are several known characterizations of the elements in $\cD_\cV(\cF)$ (see for example \cite{CE06}). 
In what follows we describe a simple {\it parametrization} of $\cD_\cV(\cF)$ in terms of $\cD_\cW(\cF)$ i.e. the classical dual frames for $\cF$ in $\cW$, which is implicit in \cite{CE06}.

\begin{pro}\label{repre dual v}
Let $\cV$ and $\cW$ be closed subspaces of $\cH$ such that 
$\cW\orto \oplus \cV=\hil$.
Let $\cF=\{f_i\}_{i\in I}$ be a frame for $\cW$. Then the map 
$$ 
\cD(\cF)\ni \{g_i\}_{i\in I} \mapsto 
\{P_{\cV// \cW^\perp} \, g_i\}_{i\in I}\in \cD_\cV(\cF)
$$ 
is a (linear) bijection between $\cD(\cF)$ and $\cD_\cV(\cF)$ that sends $\cF^\#$ to $\cF_\cV^\#$.
\end{pro}
\begin{proof}
Let $\cG=\{g_i\}_{i\in I}\in \cD(\cF)$ and set $\cG'=
\{P_{\cV// \cW^\perp}\, g_i\}_{i\in I}\,$. 
Then $T_{\cG'}=P_{\cV// \cW^\perp} \,T_\cG$ and hence
$$
T_{\cG'}\, T_\cF^*= P_{\cV// \cW^\perp} \, T_\cG \,T_\cF^*
=P_{\cV// \cW^\perp} \, P_\cW=P_{\cV// \cW^\perp}\ .
$$ 
Therefore $\cG'\in \cD_\cV(\cF)$ and the map is well defined.
To check that the map is injective, let $\cG=\{g_i\}_{i\in I} $ and $  \cK=\{k_i\}_{i\in I} \in \cD(\cF)$ be such that $P_{\cV// \cW^\perp} g_i= P_{\cV// \cW^\perp} k_i$ for $i\in I$. Then, 
$$
P_{\cV//\cW^{\perp}}T_{\G}=P_{\cV//\cW^{\perp}}T_{\cK}  \implies 
P_{\cV//\cW^{\perp}}(T_{\cG}-T_{\cK})=0  \implies 
R(T_{\cG}-T_{\cK})\inc \cW^\perp \ . 
$$
But also $R(T_{\cG}-T_{\cK})\inc \cW$, so $R(T_{\cG}-T_{\cK})=\{0\}$ and $T_{\cG} =T_{\cK} \,$.

\pausa
Finally we check that the map is surjective. Recall that, 
since $\cW\orto \oplus \cV=\hil$, then the map 
 $P_{\cV//\cW^\perp}|_{\cW}:\cW\rightarrow \cV$ is a linear bounded isomorphism.
Thus, given $\cK=\{k_i\}_{i\in I}\in \cD_\cV(\cF)$ there exists a unique Bessel sequence $\cG=\{g_i\}_{i\in I}$ in $\cW$ such that $P_{\cV//\cW^\perp} g_i=k_i$ for $i\in I$. Then, $P_{\cV//\cW^\perp} T_\cG=T_\cK$ and therefore
$$ 
P_{\cV//\cW^\perp}=T_\cK\,T_\cF^*= P_{\cV//\cW^\perp} T_\cG \, T_\cF^* \ \
\implies \ \ P_{\cV//\cW^\perp} (T_\cG \, T_\cF^* - P_\cW)=0\ .
$$
Since $R(T_\cG \, T_\cF^* - P_\cW)\subseteq \cW$ then previous equation implies that $T_\cG \, T_\cF^* - P_\cW=0$ and hence $\cG\in \cD(\cF)$ is such that $\{P_{\cV//\cW^\perp} \,g_i \}_{i\in I}=\cK$.
\end{proof}

\pausa
The previous result allows to obtain several other representations of the $\cV$-duals of $\cF$ from the classical theory of dual frames for $\cF$ in $\cW$. The following result is an example of this phenomenon

\begin{cor}\label{cororepre}
Let $\cV$ and $\cW$ be closed subspaces of $\cH$ such that $\cW\orto \oplus \cV=\hil$.
Let $\cF=\{f_i\}_{i\in I}$ be a frame for $\cW$ with canonical $\cV$-dual frame $\cF^\#_\cV=\{f^\#_{\cV,\,i}\}_{i\in I}$ 
defined in  Eq. \eqref{canVdual}. 
Given any  $\cG 
\in \cV^{\, I} $, then $\cG\in \cD_\cV(\cF) \iff $ there exists a Bessel sequence 
$\cZ=\{z_i\}_{i\in I}\in \cV^{\, I}$ such that
$$ 
T_\cZ \, T_\cF^* \, f = \sum_{i\in I} \langle f\coma f_i\rangle \ z_i=0 \peso{for every} 
f\in \hil \peso{and} \cG=\{f^\#_{\cV,\,i}+z_i\}_{i\in I}\ .$$
\end{cor}
\begin{proof}
Let $\cK=\{k_i\}_{i\in I}\in \cD(\cF)$ be such that $\cG=\{P_{\cV//\cW^\perp}\,k_i\}_{i\in I}$ as in Proposition \ref{repre dual v}. Since  $\cK\in \cD(\cF)$, it is well known that there exists a Bessel sequence $\cX=\{x_i\}_{i\in I}$ in $\cW$ such that $\cK=\{f_i^\#+x_i\}_{i\in I}$ and such that $T_\cX\,T_\cF^*=0$, where $T_\cX$ denotes the synthesis operator of $\cX$ (see for example \cite{Chrisbook}).
Hence,
$\cG=\{P_{\cV//\cW^\perp}(f_i^\#+x_i)\}_{i\in I}$ which shows that $T_\cG=T_{\cF^\#_\cV}+ P_{\cV//\cW^\perp} T_\cX$ with $(P_{\cV//\cW^\perp} T_\cX)T_\cF^*=0$ and the result holds for $\cZ=\{P_{\cV//\cW^\perp} x_i\}_{i\in I}\in \cV$. The converse is straightforward. 
\end{proof}

\pausa
From now on, we shall restrict our attention to finite sequences of vectors in $\hil$; accordingly, we shall consider
decompositions $\cW^\perp\oplus\cV=\cH$, where
$\cV$ and $\cW$ are finite dimensional subspaces of the Hilbert space $\cH$.

\pausa
In what follows, we shall be concerned with the spectral properties of frame operators of $\cV$-duals of $\cF$. Thus, we introduce some convenient notations.
\begin{fed}\label{defn all dual frames}\rm
Let $\cV$ and $\cW$ be finite dimensional subspaces of the Hilbert space $\cH$ such that $\cW^\perp\oplus\cV=\cH$.  
Let 
$\cF=\{f_i\}_{i\in \In}$ be a frame for $\cW$.  We consider
$$
\cS \cD_\cV(\cF) \igdef
\{S_\cG=T_\cG\,T_\cG^*: \ \cG\in\cD_\cV(\cF)\}\subset \posop\ . $$
the set of frame operators of $\cV$-dual frames of $\cF$. \EOE
\end{fed}

\begin{pro}\label{pro tec}
Let $\cV$ and $\cW$ be finite dimensional subspaces of the Hilbert space $\cH$ such that $\cW^\perp\oplus\cV=\cH$ and
let $\dim \cV=\dim \cW=d$. Let $\cF=\{f_i\}_{i\in \In}$ be a frame for $\cW$. Then,
$$
\cS \cD _\cV(\cF)=\left\{S_{\F^\#_\cV}+B: \   \ B\in\posop\ , \ R(B)\subseteq\cV \py \rk \, B\leq n-d \right\}.$$
\end{pro}
\begin{proof}
Given $\cG\in\cD_\cV(\cF)$, Corollary \ref{cororepre} shows that there exists a Bessel sequence $\cZ=\{z_i\}_{i\in \In}$ 
in $\cV$ such that
$T_\cG=T_{\cF_\cV^\#}+T_\cZ$ and $T_\cZ \,T_\cF^*=0$. Notice that 
$T_{\cF_\cV^\#}=P_{\cV//\cW^\perp} \, T_{\cF^\#}=P_{\cV//\cW^\perp} \, S_\cF^\dagger \,T_\cF$ 
which implies that $T_\cZ \,T_{\cF_\cV^\#}^*=0$. Hence,
$$
S_\cG=T_\cG\, T_\cG^*=(T_{\cF^\#_\cV}+T_\cZ)(T_{\cF^\#_\cV}+T_\cZ)^*=S_{\cF^\#_\cV}+S_{\cZ}\, ,   $$
where $S_\cZ\in \posop$ is the frame operator of $\cZ$, which is a finite rank operator. 
Since $T_\cZ T_\cF^*=0$  then $\dim \ker T_\cZ\geq d$.
Therefore,  $R(S_\cZ)=R(T_\cZ)$ so that $R(S_\cZ)\subset \cV$ and $\rk \, S_\cZ =\rk \, T_\cZ\leq n-d$.

\pausa
Conversely, let $B\in\posop$ be such that $R(B)\subseteq \cV$ and $\text{rk}(B)\le n-d$. 
Then, there exists $Z \in B(\C^n 
\coma \cV)$, such that $Z\, T_\cF^*=0$ and $B = ZZ^*$: indeed, since $\dim (R(T_\cF^*)^\perp)=n-d$ there exists a partial isometry $W\in 
B(\C^n 
\coma \cV)$ 
with initial space $\ker W ^\perp\subset R(T_\cF^*)^\perp$ and final space $R(B)=R(B^{1/2})$ so that $Z=B^{1/2}W$ has the desired properties.
If we let $\{e_i\}_{i\in\In}$ denote the canonical basis of $\C^n$ 
and $\cG=\{(T_{\cF^\#_\cV}+Z)e_i\}_{i\in\In}$ 
then $\cG$ is a finite sequence in $\cV$ such that $T_\cG=T_{\cF^\#_\cV}+Z$ so that $$
T_\cG\,T_\cF^*=T_{\cF^\#_\cV}\,T_\cF^*=P_{\cV//\cW^\perp} \ . 
$$
Hence $\cG\in \cD_\cV(\cF)$ and $S_\cG= S_{\cF^\#_\cV}+ZZ^*=S_{\cF^\#_\cV}+B$, since $ZT_{\cF^\#_\cV}^*=0$.
\end{proof}

\begin{rem}\label{carac espect U}
Let $A_0\in \matpos$ and consider an integer $m<d$. Define 
\beq\label{def U_0}
U(A_0\coma m)\igdef \{A_0+C:\ C\in \matpos\, , \ \rk \,C\leq d-m \ \}\ .
\eeq
We point out that the spectral structure of the set $U(A_0\coma m)$ is described in \cite{MRS13}. Indeed, given $\mu\in(\R^d)^\downarrow$ then there exists $A=A_0+C\in U(A_0\coma m)$ such that $\la(A)=\mu$ (i.e. the eigenvalues of $A$, counting multiplicities and arranged in non-increasing order, coincide with the entries of $\mu$) if and only if
\begin{enumerate}
\item $\mu_i\geq \lambda_i(A_0)$ for $i\in \IN{d}\,$, in case $m\leq 0$;
\item $\mu_i\geq \lambda_i(A_0)$ for $i\in \IN{d}$
and $\mu_{d-m+i}\leq \lambda_i(A_0)$ for $i\in \IN{m}\,$, in case $m\geq 1$.
\EOE
\end{enumerate}
\end{rem}

\pausa
Recall from Eq. \eqref{barraM} that $\M\subseteq \N$ stands for $\M=\IN{p}$ or $\M=\N$ in such a way that $\dim\hil=|\M|$. 
Henceforth, $\ell^1_+(\M)^\downarrow$ denotes the space of sequences $\lambda=(\lambda_i)_{i\in\M}$ 
with $\lambda_i\geq \lambda_j\geq 0$ for $i,\,j\in\M$ such that $i\leq j$ and
$\tr(\lambda)\igdef \sum_{i\in\M}\lambda_i<\infty$. 

\pausa 
 Let $\cF=\{f_i\}_{i\in\In}\in\hil^n$ be a finite sequence with frame operator $S_\cF\in B(\hil)^+$. Hence, $S_\cF$ is a positive semidefinite finite rank operator, with range $\cW=$ Span$\{f_i:\ i\in\In\}\subset \hil$. Let $d=\dim\cW$, and let 
 $(S_\cF)_\cW\in B(\cW)^+$ be the compression of $S_\cF$ to $\cW$ (see Remark \ref{algo sobre compresion}); Then, we define 
 $$\la(S_\cF)=( (\la_i((S_\cF)_\cW))_{i\in\IN{d}},0_{|\M|-d})\in \ell^1_+(\M)\da\ ,$$ where $(\lambda_i((S_\cF)_\cW))_{i\in\IN{d}}\in \R_{>0}^d$ denotes the vector of eigenvalues of the compression $(S_\cF)_\cW\in B(\cW)^+$, counting multiplicities and arranged in non-increasing order. It is straightforward to check that 
 $\la(S_\cF)\in\ell^1_+(\M)\da$ coincides with the vector of singular values (or s-numbers) of the compact operator $S_\cF\in B(\hil)^+$ (see \cite{BSi}).

\begin{teo}[Spectral structure of $\cV$-duals]\label{teo struc espec}
Let $\cV$ and $\cW$ be finite dimensional subspaces of the Hilbert space $\cH$ such that $\cW^\perp\oplus\cV=\cH$ and
let $\dim \cV=\dim \cW=d$. Let $\cF=\{f_i\}_{i\in \In}$ be a frame for $\cW$ and denote $\lambda(S_{\cF^\#_\cV})=\lambda_{\cV}^\#=(\lambda_{\cV,\,j}^\#)_{j\in\M}$ and $m=2d-n$. Given $\mu=(\mu_i)_{i\in\M}\in 
\ell^1_+(\M)^\downarrow$,
the following conditions are equivalent:
\ben
\item There exists $\cG\in\cD_\cV(\cF)$ such that $\lambda(S_\cG)=\mu$;
\item $\mu_i=0$ for $i\geq d+1$ and:
\ben
\item $\mu_i\geqp \lambda_{\cV,\, i}^\#$ for $i\in \IN{d}\,$, in case $m\leq 0$;
\item $\mu_i\geqp \lambda_{\cV,\, i}^\#$ for $i\in \IN{d}$
and $\mu_{d-m+i}\le \la_{\cV,\, i}^\#$ for $i\in \IN{m}$, in case $m\geq 1$.
\een
\een
\end{teo}
\begin{proof}

Fix an ONB $\{v_i\}_{i\in\IN{d}}$ of $\cV$. Let $A_0\in\matpos$ be given by 
$A_0=(\langle S_{\cF^\#_\cV} v_j\coma v_i\rangle)_{i,\,j\in\IN{d}}$ and let 
$m=2d-n$ (so that $d-m=n-d$). Then 
$\la(S_{\cF^\#_\cV})
=(\lambda(A_0)\coma 0_{|\M|-d})\in(\ell_+^1(\M))\da$.
Using Proposition \ref{pro tec}, to each $S_\cG= S_{\F^\#_\cV}+B \in \cS\cD_\cV(\cF)$ 
we can associate the element $A_0+C\in U(A_0\coma m)$ where $C=(\, \langle B v_j\coma v_i\rangle\,)_{i,\,j\in\IN{d}}\in\matpos$ 
in such a way that 
$$
\lambda(S_\cG)=(\lambda(A_0+C)\coma 0_{|\M|-d})\in \ell_+^1(\M)\da\ . 
$$
Conversely, if $A_0+C\in U(A_0\coma m)$ then there exists $\cG\in \cD_\cV(\cF)$ such that the 
matrix corresponding to $S_\cG$ as above is $A_0+C$. Thus, the previous remarks show that
\beq \label{ident conj6}
\{\la(S_\cG):\ \cG\in \cD_\cV(\cF)\}=\{(\la(A)\coma 0_{|\M|-d}): \ A\in U(A_0\coma m)\}\ .
\eeq
The proof now follows from Eq. \eqref{ident conj6} above and Remark \ref{carac espect U}.
\end{proof}

\begin{rem}
Using Theorem \ref{teo struc espec}, in case $\cV=\cW$ (i.e. classical duality) we recover the structure of classical duals of a frame $\cF$ for the Hilbert space $\cW$ as described in \cite{MRS13}.\EOE
\end{rem}

\begin{cor}\label{hay nu-dual parseval}
There exists $\cG\in\cD_\cV(\cF)$ which is Parseval in $\cV$ if and only if:
\ben
\item[(a)] $1 \geq \la_{\cV,\, i}^\#$ for $i\in \IN{d}\,$, in case $m=2d-n\leq 0$;
\item[(b)] $1 \geq \la_{\cV,\, i}^\#$ for $i\in \IN{d}\,$ and $1 = \la_{\cV,\, i}^\#$ for $i\in \IN{m}$,
in case $d-1\geq m=2d-n\geq 1$.
\een
\end{cor}
\begin{proof}
Let $\cG$ be a frame for $\cV$. Notice that $\cG$ is a Parseval in $\cV$, i.e. $S_\cG=P_\cV\,$, 
if and only if $\lambda_i(S_\cG)=1$ for every $ i\in\IN{d}\,$. Thus, the result now follows from Theorem \ref{teo struc espec}.
\end{proof}

\begin{rem}\label{rel dhan}
With the notations and terminology from Theorem \ref{teo struc espec}, notice that Corollary \ref{hay nu-dual parseval} can be written as follows: there exists $\cG\in\cD_\cV(\cF)$ which is Parseval in $\cV$ if and only if
$$ S_{\cF^\#_\cV}\leq P_\cV \peso{and} \dim R(P_\cV-S_{\cF^\#_\cV})\leq d-m=n-d= \dim\ker T_\cF \ .$$  This last formulation of the existence of Parseval $\cV$-duals formally resembles the characterization in 
\cite[Proposition 2.4]{dhan} in case of classical duality. \EOE
\end{rem}

\subsection{Optimal oblique duals with norm restrictions}

In applied situations, it is desired to characterize the existence (and find explicit methods of construction) of frames with some prescribed parameters. This kind of problems are referred to as {\it frame design problems}, and they are at the core of finite frame theory (see for example \cite{BF,CKFT,FMP,MR10,MRS13,MRS14,Pot} and the recent book \cite{FF}). 

\pausa 
Let $\cW,\,\cV\subset \hil$ be two finite dimensional subspaces such that  $\cW \oplus\cV\orto=\hil$, and let $\dim\cW=\dim\cV=d$. Given a fixed frame $\cF=\{f_i\}_{i\in \In}\in\cW^n$ for $\cW$
we can ask whether there exists $\cG\in \cD_\cV(\cF)$ with some prescribed parameters; and in case such a dual exists we would like to obtain a procedure to construct it. For example, given $\mu\in (\R_{\geq 0}^d)^\downarrow$ we can ask whether there exists $\cG\in \cD_\cV(\cF)$ with $\lambda(S_\cG)=\mu$.
Notice that Theorem \ref{teo struc espec} above completely solves this problem; moreover, the proof of Proposition \ref{pro tec} contains a procedure to effectively construct such a dual.

\pausa
As a consequence of the description  of the spectra of elements in $\cS(\cD_\cV(\cF))$, we see that if $\cG\in \cD_\cV(\cF)$ then $S_\cG\geq S_{\cF_\cV^\#}\,$. This last fact implies that the canonical $\cV$-dual is optimal with respect to several criteria (including convex potentials). Yet, from a numerical point of view the oblique canonical $\cV$-dual might not be the best choice of a $\cV$-dual for $\cF$. For example, the condition number of the frame operator $S_{\cF^\#_\cV}$ may not be minimal in $\cD_\cV(\cF)$; indeed, Corollary \ref{hay nu-dual parseval} shows that under certain assumptions we can consider a Parseval $\cV$-dual of $\cF$ (with minimal condition number).

\pausa
In order to search for alternate $\cV$-duals that are numerically  robust, we proceed as follows: for $t\geq \tr(S_{\cF^\#_\cV})$ we consider
$$
\cD_{\cV, \, t}(\cF)\igdef\big\{\,\cG=\{g_i\}_{i\in\In}\in \cD_\cV(\cF): \ \sum_{i\in\In}\|g_i\|^2\geq t \,\big\}\ .  
$$
Notice that if $t>\tr(S_{\cF^\#_\cV})$ then the canonical $\cV$-dual is not in
$\cD_{\cV, \,t}(\cF)$ and therefore, it is natural to ask whether there is an optimal dual in $\cD_{\cV,\,t}(\cF)$. Using the well known identity 
\beq\label{eq ident tra norm}
 \sum_{i\in\In} \|g_i\|^2= \tr(S_\cG)= \sum_{i\in\IN{d}}\mu_i  \ ,
 \eeq
  where $\lambda(S_\cG)=\mu$, we see that Theorem \ref{teo struc espec} gives a complete solution to a frame design problem in the sense that it allows to get a complete description of the eigenvalue lists of the frame operators of elements in $\cD_{\cV, \,t}(\cF)$.

\def\nulamt{\nu_{\la \coma m}(t)}
\def\numumt{\nu_{\mu \coma m}(t)}
\begin{rem}\label{rem estruc espec Ut}
Let $A_0\in \matpos$, $t\geq \tr(A_0)\geq 0$ and consider an integer $m<d$. Define 
$$
U_t(A_0\coma m) \igdef \{A_0+C:\ C\in \matpos\, , \ \rk(C)\leq d-m\, , \ \tr(A_0+C)\geq t\}\ .
$$
The spectral and geometrical structure of the set $U_t(A_0\coma m)$ is described 
in \cite{MRS13}; in particular, there it is shown that there exist $\prec_w$-minimizers within this set. 
Indeed, using the previous notations, if $\lambda(A_0)=\la=(\lambda_i)_{i\in \IN{d}}\in(\R_{\geq 0}^d)\da$, we consider $h_m:[\lambda_d\coma \infty)\rightarrow \R_{\geq 0}$ given by 
$$h_m(t)=\sum_{i=\max\{m\,,\,0\}+1}^d (t-\lambda_i)^+\ , $$
 where $x^+$ stands for the positive part of $x$. Notice that $h_m$ is strictly increasing; hence there exists a unique $c_{\la,\,m}(t)=c  \geq \la_d$ such that $h_m(c)=t-\tr \la$. Then, set
\ben
\item $\nulamt\igdef \big(\,(c-\la_1)^+ +\la_1\coma \ldots\coma 
(c-\la_d)^+ +\la_d\,\big) \in (\R^d_{\geq 0})\da$, if $m\leq 0$;
\item $\nulamt \igdef(\la_1\coma \ldots \coma \la_m \coma (c-\la_{m+1})^++\la_{m+1}\coma \ldots\coma (c-\la_{d})^+ +\la_{d})
\in \R_{\geq 0}^d\,$, if $m\in\IN{d-1}\,$. 
 \een
Notice that if $t>\tr(A_0)$ then $ \nulamt \in \R_{>0}^d\,$. Then, it turns out that (see \cite{MRS13}) 
\ben
\item There exists $A^{\rm op}\in U_t(A_0 \coma m)$ such that $\la(A^{\rm op})=\nulamt \da$;
\item For every $A\in U_t(A_0\coma m)$ then $\nulamt \prec_w\la(A)$;
\item If $A=A_0+B\in U_t(A_0\coma m)$ then $\la(A)=\nulamt $ if and only if  
$\nulamt = \la (A_0) + \la\ua(B)$ and 
there exists an ONB 
$\{z_i\}_{i\in\IN{d}}$ of $\C^d$ such that 
\beq
A_0=\sum_{i\in \IN{d}} \la_i\ z_i\otimes z_i \peso{ and } B=\sum_{i\in \IN{d}} \la_{d-i+1}(B)\ z_i\otimes z_i \ . \EOEP
\eeq
\een
\end{rem}

\pausa
The following result shows that there are structural minimizers of arbitrary (strictly) convex potentials
in $\cD_{\cV, \,t}(\cF)$, i.e. duals $\cG\in\cD_{\cV, \,t}(\cF)$ that simultaneously minimize every convex potential. This is interesting from an applied point of view, since evaluations of convex potentials (e.g. the frame potential as described in Eq. \eqref{rep BFpot}) are typically easier to compute than structural parameters (i.e. computing eigenvalue lists or eigenvectors)

\begin{teo}[Optimal duals in $\cD_{\cV, \,t}(\cF)$]\label{sobre nuduales optimos}
Let $\cV$ and $\cW$ be finite dimensional subspaces of the Hilbert space $\cH$ such that $\cW^\perp\oplus\cV=\cH$ and
let $d=\dim \cV=\dim \cW$. Let $\cF=\{f_i\}_{i\in\In}$ be a frame for $\cW$ and set $\lambda_{\cV}^\#
\igdef \lambda(S_{\cF^\#_\cV})$. For every  $t\geq \tr \lambda_{\cV}^\#$ 
there exists $\nu\in\ell^1_+(\M)^\downarrow$ %such that
with the following minimality properties: 
\ben
\item There exist $\cG_{\rm op}\in \cD_{\cV,\,t}(\cF)$ such that $\lambda(S_{\cG_{\rm op}})=\nu$;
\item For every non-decreasing function $h\in\convf$ then
\begin{equation}\label{inequa op}P_h(\cG_{\rm op})\leq P_h(\cG) \ \ \ , \ \ \cG=\{g_i\}_{i\in\In}\in\cD_{\cV,\,t}(\cF)\ .\end{equation}
\een
Moreover, if we assume further that $h\in\convfs$ and $\cG=\{g_i\}_{i\in\In}\in\cD_{\cV,\,t}(\cF)$ attains equality in \eqref{inequa op}, then $\lambda(S_\cG)=\nu$ and there exists an ONB $\{x_i\}_{i\in\IN{d}}$ for $\cV$ 
such that $$
S_{\cF^\#_\cV}=\sum_{i\in\IN{d}} \la_{\cV,\,i}^\#\ x_i\otimes x_i \peso{and}
B=S_\cG-S_{\cF^\#_\cV}=\sum_{i\in\IN{d}} \lambda_{d-i+1}(B) \ x_i\otimes x_i\ . $$
\end{teo}
\begin{proof} 
Fix $\{v_i\}_{i\in\IN{d}}$ an ONB of $\cV$ and let $A_0\in\matpos$ be given by 
$A_0=\big(\,\langle S_{\cF^\#_\cV}\, v_j \coma v_i\rangle\,\big)_{i,\,j\in\IN{d}}\,$. 
Arguing as in the proof of Theorem \ref{teo struc espec}, and taking into account the identity in Eq. \eqref{eq ident tra norm} we see that
\beq \label{ident conj7}
\{\la(S_\cG):\ \cG\in \cD_{\cV,\,t}(\cF)\}=\{(\la(A)\coma 0_{|\M|-d}): \ A\in U_t(A_0\coma m)\}\ .
\eeq
Set $\la=\la(A_0)=(\la^\#_{\cV,\,i})_{i\in\IN{d}}\in \R_{>0}^d$, and notice that $t\geq \tr\la$; set $m=2d-n$ and consider 
$\nu_{\la,\,m}(t)\in \R_{>0}^d$ defined as in Remark \ref{rem estruc espec Ut}. Finally, define $$\nu=(\nu_{\la,\,m}(t),0_{|\M|-d})\in \ell^1_+(\M)\ . $$
If $h\in\convf$ is a non-decreasing function then, by Definition \ref{pot generales}, we get that
\beq\label{ec defi poth}
P_h(\cG)=\sum_{i\in\IN{d}}h(\la_i(S_\cG))=\tr((S_\cG)_\cV) \peso{for} \cG\in\cD_\cV(\cF)\ .
\eeq 
Hence, the proof now follows from Eqs. \eqref{ident conj7} and \eqref{ec defi poth} above, Remark \ref{rem estruc espec Ut} and the relation between submajorization and non-decreasing convex functions described in Section \ref{sec 2.2.}.
\end{proof}

\section{Optimal $(\cV,\cW)$-oblique dual pairs with prescribed
parameters}\label{sec optimal rot1}

It has long been recognized that for a fixed frame $\cF$ for $\cW$, oblique $\cV$-duality offers a much more flexible theory than classical duality, which comes from the fact that we can choose $\cV$ from a large class of subspaces (see \cite{CE06} for example). Moreover, it has also been noticed that the relative position of the subspaces $\cV$ and $\cW$ for which $\cW\oplus \cV^\perp=\hil$ plays a key role when comparing oblique duality to classical duality. This phenomenon has been studied mainly though the angle between the subspaces $\cV$ and $\cW$ (see \cite{YEldar3,CE06,YEldar1} and the definitions below).
Yet, the angle between the subspaces $\cV$ and $\cW$ only provides qualitative measure of the role of the relative geometry of $\cV$ and $\cW$ in the context of $\cV$-duality. In what follows, we give a detailed description of the role of the relative position of $\cV$ and $\cW$ in the $\cV$-duality of $\cF$ in case the subspaces are finite dimensional. Our analysis relies on a multiplicative Lidskii's inequality and it is based on the complete list of the so-called principal angles between $\cV$ and $\cW$. Our results provide sharp quantitative measures of these relations.

\subsection{Relative geometry between finite dimensional subspaces}

We begin by describing the principal angles and vectors between finite dimensional subspaces. 
Let $\cV,\,\cW\inc\hil$ be finite dimensional subspaces with
$\dim\cV=\dim\cW=d$. 
Let  $P_\cV$ and $P_\cW$ denote the orthogonal projections onto $\cV$ and $\cW$ respectively. The principal angles 
$$
0\leq \theta_1\leq \ldots\leq \theta_d\leq \frac{\pi}{2} 
$$
are defined  (see \cite{GoLo,KA}) in such a way that the 
positive finite rank operator $|P_\cW\, P_\cV|\in \op$ satisfies that 
$$ 
\lambda(|P_\cW\, P_\cV|)=(\cos(\theta_1)\coma \ldots\coma \cos(\theta_d)\coma 0_{|\M|-d})\in \ell^1_+(\M)^\downarrow\  . 
$$
We say that $w_1\coma \ldots\coma w_d\in \cW$ (respectively $v_1\coma \ldots\coma v_d\in \cV$) are principal vectors (or principal directions) between $\cV$ and $\cW$ if they are an o.n. basis of $\cW$ (respectively if they are an o.n. basis of $\cV$) such that 	
\beq\label{cosenos} 
| P_\cV\,P_\cW |\, w_i=\cos(\theta_i)\,w_i \ \ \ (\ \  \text{respectively}   |P_\cW\,P_\cV|\, v_i=\cos(\theta_i)\,v_i \ \  ) 
\peso{for every} i \in \IN{d} 
\ . 
\eeq
An alternative characterization of principal angles and
vectors is as follows: given $k\in \IN{d}\,$, then define inductively
$$
\langle v_k \coma w_k\rangle 
=\cos(\theta_k)=\max_{v\in \cV} \ \max_{w\in \cW} \ \ \langle v\coma  w\rangle  
$$ 
subject to the restrictions 
$$ 
\|v\|=\|w\|=1 \ \ , \ \  \langle v,\,v_i\rangle =0 \py  \langle w\coma w_i\rangle =0 \peso{for} \ 0\leq i\leq k-1\ , 
$$ 
where we set $v_0=w_0=0$. Notice that the principal angles between $\cV$ and $\cW$ are a qualitative measure of the relative position between these subspaces.

\pausa
Assume further that $\cW^\perp  \oplus\cV=\hil$. Consider the oblique projection $P_{\cV//\cW^\perp}\,$. 
In this case, there exists a connection between the principal angles and vectors between $\cV$ and 
$\cW$ and the geometrical and spectral structure of $P_{\cV//\cW^\perp}\,$. Indeed, it is known 
(see \cite{CoMa10} and the references therein) that the Moore-Penrose pseudo-inverse of $P_{\cV//\cW^\perp}$ 
is given by 
\begin{equation}\label{ec proy ob}
(P_{\cV//\cW^\perp}) ^\dagger=P_\cW\,P_\cV \implies|P_{\cV//\cW^\perp}|^\dagger =|P_\cV\,P_\cW| \ \text{ and } \ |(P_{\cW//\cV^\perp})^*|^\dagger =|P_\cW\,P_\cV | \ . \end{equation}
In this case, since $P_{\cV//\cW^\perp}$ has rank $d$, then $|P_\cV\,P_\cW |$ also has rank $d$ and therefore $\theta_d<\pi/2$.
Moreover, by Eq. \eqref{cosenos}, the principal vectors between $\cV$ and $\cW$ 
satisfy that 
\begin{equation}\label{eq ang y vec} 
|P_{\cV//\cW^\perp}|\, w_i=\frac{1}{\cos(\theta_i)}\,w_i \py |P_{\cW//\cV^\perp} |\, v_i=\frac{1}{\cos(\theta_i)}\,v_i 
\peso{for every} i \in \IN{d} 
\ . \end{equation}
Take the polar decomposition $ P_\cV\,P_\cW  = U\, | P_\cV\,P_\cW |$
with the unique  partial isometry $U\in \op$ with initial space $\cW$ and final 
space $\cV$. Hence we have that $| P_\cW\,P_\cV | = U \, | P_\cV\,P_\cW |\, U^*$. 
Therefore, given principal vectors $\{w_i\}_{i\in\IN{d}}\in \cW^d$ between $\cV$ and $\cW$,  
in what follows 
we shall assume that the corresponding
 principal vectors $\{v_i\}_{i\in\IN{d}}\in\cV^d$ between $\cV$ and $\cW$ are given by $v_i = U\, w_i$ for 
every $i\in\IN{d}\,$. In particular, it holds that 
\beq\label{sin primas}
P_\cW \, v_i = \cos (\theta_i) \, w_i 
\py P_\cV \, w_i = \cos (\theta_i) \, v_i 
\peso{for every} i \in \IN{d}\ ,
\eeq
because, for example, 
$P_\cV \, w_i = P_\cV\,P_\cW \,w_i = 
U\, | P_\cV\,P_\cW | \, w_i = \cos (\theta_i) \ U \, w_i 
=\cos (\theta_i) \  v_i  \ . 
$

\begin{rem}[On two notions of angle between subspaces] \label{rem comp angulos} There are two different notions
of angle between subspaces in the literature. Next we include their definitions, we compare them and we also relate them to the principal angles defined above. Hence, consider
two finite dimensional subspaces $\cV,\,\cW\inc\hil$  with
$\dim\cV=\dim\cW=d$. Let $(\theta_j)_{j\in\IN{d}}$ denote the principal angles between $\cV$ and $\cW$. 
\ben
\item  In \cite{UnAl} the authors introduce the angle $\theta_{\cV,\,\cW}\in[0,\pi/2]$ between 
$\cV$ and $\cW$ defined by 
$$
\cos(\theta_{\cV,\,\cW}) = \inf_{f\in\cW \coma \|f\|=1} \|P_\cV f\|
\ .
$$
Therefore, 
\beq\label{defi ang UA}
\cos(\theta_{\cV,\,\cW})^2 =\inf_{f\in\cW,\,\|f\|=1} \langle |P_\cV\,P_\cW|^2f\coma f\rangle=\cos(\theta_d)^2\ . 
\eeq
That is, we have the identity $\theta_{\cV,\,\cW}=\theta_d\,$. 
If we assume further that $\cW^\perp  \oplus\cV=\hil$ then Eqs. \eqref{eq ang y vec} and \eqref{defi ang UA} provide a simple proof of the identity $\|P_{\cV//\cW^\perp}\|= \cos(\theta_{\cV,\,\cW})^{-1}$. 

\item There is yet another notion of angle between subspaces, the so-called Dixmier angle, denoted by $\theta^{\cV,\,\cW}\in[0,\pi/2]$ and given by
$$%\beq\label{def ang Dixmier}
 \cos(\theta^{\cV,\,\cW})=\sup_{v\in\cV,\, w\in\cW,\,\|v\|=\|w\|=1} |\langle v,\,w\rangle| =\|P_\cV\, P_\cW\|= \cos(\theta_1) \ . 
$$ %\eeq
That is, we have the identity $\theta^{\cV,\,\cW}=\theta_1\,$. If we assume further that $\cW^\perp  \oplus\cV=\hil$ then it is well known (see \cite{Deu}) that 
$\|P_{\cV//\cW^\perp}\|=\sin(\theta^{\cV,\,\cW^\perp})^{-1}$ which implies that $\sin(\theta^{\cV,\,\cW^\perp})=\cos(\theta_{\cV,\,\cW})$
and hence we get that 
\begin{equation*}
\theta_{\cV,\,\cW}=\pi/2-\theta^{\cV,\,\cW^\perp}\ .  \EOEP 
\end{equation*}
\een
\end{rem}

\subsection{Optimal dual pairs by rigid rotations}

We begin by fixing the following notations:

\begin{nota}\label{notas}
Throughout the rest of the paper we shall consider 
\ben
\item $\cV,\,\cW\subset \hil$ two finite dimensional subspaces such that $\cV\oplus\cW\orto=\hil$;
\item $\angle(\cV;\cW)=(\theta_j)_{j\in \IN{d}}\in ([0,\pi/2)^d)^\uparrow$ principal angles ($\dim\cV=\dim\cW=d$); 
\item $\{v_j\}_{j\in \IN{d}}\in\cV^d$ , $\{w_j\}_{j\in \IN{d}}\in\cW^d$ principal vectors
in $\cV$ and $\cW$ obeying Eq. \eqref{sin primas};
\item $\cF=\{f_i\}_{i\in \In}\in\cW^n$ a frame for $\cW$ with 
\beq
\lambda(S_\cF)=\la=(\la_i)_{i\in\M} 
\py \la(S_{\cF^\#_\cV})=\la^\#_\cV=(\la^\#_{\cV,\,i})_{i\in\M}\ .
\EOEP
\eeq
\een
\end{nota}
\noi Consider the Notations \ref{notas}. 
In order to have an estimate of $\lambda^\#_\cV$ notice that
 $$\cF^\#_\cV =P_{\cV//\cW^\perp} \cdot \cF^\# \implies T_{\cF^\#_\cV}=P_{\cV//\cW^\perp} \ S_\cF^\dagger \ T_\cF$$
and hence
\begin{equation}\label{eq estimacion lambda}
 S_{\cF^\#_\cV} = P_{\cV//\cW^\perp} \, S_\cF^\dagger \, (P_{\cV//\cW^\perp})^*=P_{\cV//\cW^\perp} \, S_\cF^\dagger \, P_{\cW//\cV^\perp}\ .
 \end{equation}
The previous remarks, together with Lidskii's multiplicative inequalities in Theorem \ref{teo hay max y min mayo}, allow us to obtain the following bounds in terms of the spectral structure of $S_\cF$ and the principal angles (i.e. the relative geometry) between $\cV$ and $\cW$. We point out that the bounds given in next result are a quantitative measure of how the relative geometry of the subspaces $\cV$ and $\cW$ impact in oblique duality. 

\begin{teo}\label{teo acot lids multi}
Consider the Notations \ref{notas}. 
If we let $\mu= \big( \, \lambda_{d-j+1}^{-1}\,\cos^{-2}(\theta_j)\, \big )_{j\in\IN{d}}^\downarrow$ then
\begin{equation}\label{eq acot lids multi}
\prod_{j\in \IN{k}} \mu_j
\leq \prod_{j\in \IN{k}}\lambda^\#_{\cV,\,j}\leq \left( \prod_{j=d-k+1}^d \lambda_{j}\,\cos^2(\theta_j)\right)^{-1}\ , \ \ k\in \IN{d}\ .
\end{equation}
Moreover, $\la^\#_\cV=\mu$ (resp. 
$\la^\#_\cV =\big(\,\la_{d-j+1}\inv \cos^{-2}(\theta_{d-j+1}) \,\big)_{j\in\IN{d}}\,) 
\iff $
there exist an o.n.b. in $\cW$ of principal vectors $\{w_i\}_{i\in\IN{d}}$ between $\cV$ and $\cW$ such that $$ S_\cF=\sum_{j\in\IN{d}} \la_{d-j+1}(S_\cF)\ w_i\otimes w_i$$ 
(resp.  $S_\cF=\sum_{j\in\IN{d}} \la_{j}(S_\cF)\ w_i\otimes w_i$).
\end{teo}
\begin{proof}
Consider the representation of $S_{\cF^\#_\cV}$ given in Eq. \eqref{eq estimacion lambda}. 
Denote by $M = |P_{\cV//\cW^\perp}|$ and let $P_{\cV//\cW^\perp}=V\, M$ be the 
polar decomposition. Notice  that $R(M) 
=\cW$; hence, $M$ 
restricted to (the reducing subspace) $\cW$ is invertible.  On the other hand, since $R(S_{\cF^\#})=\cW$ then the restriction of $S_{\cF^\#}=S_\cF^\dagger$  
to (the reducing subspace) $\cW$ is also invertible. Hence, Eq. \eqref{eq estimacion lambda} implies that 
\beq\label{el s vdual2a}
 S_{ \cF^\#_\cV}=V\, ( \, M \, S_{\cF}^\dagger \, M 
 \,)\ V^* .
  \eeq
Since $V$ is a partial isometry with initial space $\cW$ and final space $\cV$ then, Eq. \eqref{el s vdual2a} 
implies that  
\beq\label{el s vdual3a}
\la(S_{ \cF^\#_\cV})=\big(\, \la \left(M_\cW \, 
(\,S_{\cF}^\dagger\,)_\cW\ M_\cW 
\right)
\coma 0_{|\M|-d}\,\big)\ , 
\eeq where in general $S_\cW\in B(\cW)$ denotes the restriction of $S$ to its reducing subspace $\cW$.
Since $\lambda \big(\,(S_\cF^\dagger)_\cW\,\big)=(\lambda_d^{-1},\ldots,\lambda_1^{-1})$ and
$\la(M_\cW) \stackrel{\eqref{eq ang y vec}}{=} 
(\cos(\theta_d)^{-1},\ldots,\cos(\theta_1)^{-1})$, 
% (by Eq. \eqref{eq ang y vec}) 
we see that the result is now a 
 consequence of Theorem \ref{teo hay max y min mayo} (Lidskii's multiplicative inequalities)
and the definition of log-majorization.
\end{proof}

\pausa
Consider the Notations \ref{notas}. The previous result suggests that we could take advantage of the relative geometry 
between the subspaces $\cV$ and $\cW$ to construct optimal oblique encoding-decoding schemes with prescribed properties. Indeed, 
let $U\in\op$ be a unitary operator such that $U(\cW)\subset \cW$ i.e. $\cW$ is $U$-invariant. Hence, we could consider the frame $U \cdot \cF=\{U\,f_j\}_{j\in \In}$ for $\cW$. Notice that $U\cdot \cF$ preserves essentially every property of $\cF$ (e.g., linear relations, eigenvalues list of its frame operator, norms of the elements of the frame, etc). In particular, $(U\cdot \cF)^\#=U\cdot \cF^\#$ since $S_{U\cdot \cF}= U T_\cF\,T_\cF^* U^*=US_\cF U^*$; that is, the (classical) canonical dual frame of $U\cdot \cF$ in $\cW$ is the rotation by $U$ of the canonical frame for $\cF$ in $\cW$. In particular, we get that $S_{U\cdot \cF}^\dagger= US_\cF^\dagger U^*$. Nevertheless, $\cF$ and $U\cdot \cF$ can have quite different properties with respect to $\cV$-duality as shown in the following

\begin{exa}\label{ejemplo 1}
Let $\hil=\C^3$ and let $\{e_1,\,e_2,\,e_3\}$ denote the canonical basis of $\hil$. Set $\cV=\overline{\{e_2, \frac {e_1+ e_3}{\sqrt{2}}\}}$ and $\cW=  \overline{\{e_1, e_2\}}$. Notice that in this case we have that $\C^3=\cV \oplus \cW^\perp$. Set
\begin{eqnarray*}\cF_1&=&\{e_1;\, (\cos(\pi/3), \sin(\pi/3), 0)\}\subseteq \cW \ , \\
       \cF_2&=&\{e_2;\, (\cos(\pi/2+\pi/3), \sin(\pi/2+\pi/3), 0)\}\subseteq \cW\ .
            \end{eqnarray*}
Notice that $ \cF_2=U\cdot \cF_1$ where $U$ is the rotation by (the angle) $\pi/2$ in the plane $\cW$ and such that $Ue_3=e_3$. Straightforward computations show that $\lambda(S_{(\cF_1)^\#_\cV})=(8/3; 1; 0)$ and  $\lambda(S_{(\cF_2)^\#_\cV})=(3.59; 0.74; 0)$. This last fact shows that there is no unitary operator $U'$ such that $U'(\cV)=\cV$ and such that $(\cF_2)^\#_\cV=U'(\cF_1)^\#_\cV$.
\EOE
\end{exa}

\pausa
The previous example shows that the spectral properties of the $\cV$-canonical dual of a frame $U\cdot \cF$ indeed depend on $U$ and motivates the construction of unitary operators $U\in\op$ with $U(\cW)=\cW$, such that the dual pair $(U\cdot \cF,(U\cdot \cF)^\#_\cV)$ induces an optimal encoding-decoding scheme. As a measure of optimality we could consider the minimization of the joint convex potential of the pair among all such pairs; but, since the spectral properties of $U\cdot\cF$ are independent of $U$ we are left to compute those unitary operators $U_0\in\op$, with $U_0(\cW)=\cW$, that minimize - for a non-decreasing function $h\in\convf$ -  the convex potential $P_h[(U\cdot\cF)^\#_\cV]$ among all unitary operators $U\in \op$ such that $U(\cW)=\cW$. As we shall see, there exist structural solutions to this problem.

\begin{teo}\label{teo rotop} Consider the Notations \ref{notas}.
Let $\{x_i\}_{i\in \IN{d}}\in \cW^d$ be an ONB of $\cW$ 
such that $S_\cF \, x_j=\lambda_j\,x_j$ for every $ j\in \IN{d}\,$. 
\begin{enumerate}
\item Let $U_0\in\op$ be a unitary operator such that $U_0 \, x_j= w_{d-j+1}\,$ for every $  j\in \IN{d}\,$. Then
\begin{equation}
\lambda(S_{(U_0\cdot\cF)_\cV^\#})=\Big( \, (\cos^{-2}(\theta_j)\,\lambda_{d-j+1}^{-1})_{{j\in \IN{d}}}^\downarrow  \coma 0_{|\M|-d} \, \Big) \ .
\end{equation}

\item If $h\in\convf$ is non-decreasing then
\begin{equation}\label{eq desi pot h}
P_h((U_0\cdot\cF)^\#_\cV)=\min \{ P_h((U\cdot\cF)^\#_\cV): \ U\in\op\ \text{is unitary and } , \ U(\cW)=\cW\}\ .
\end{equation}
\end{enumerate}
Moreover, if we assume further that $h\in\convfs$ and $U\cdot \cF$ attains the minimum of Eq. \eqref{eq desi pot h} then there exist principal vectors $\{w_j'\}_{j\in \IN{d}}$ and a ONB $\{x'_i\}_{i\in \IN{d}}$ for $\cW$ such that $S_\cF \, x'_j=\lambda_j\,x'_j$ and 
$Ux'_j=w'_{d-j+1}\,$, for $j\in \IN{d}\,$.
\end{teo}
\begin{proof} Let $U\in\op$ be any unitary operator such that $U(\cW)=\cW$ and let $U\cdot\cF=\{Uf_i\}_{i\in\In}$. Then, notice that 
$$
(U\cdot \cF)^\#_\cV=  P_{\cV//\cW^\perp} \,(U\cdot \cF)^\#=P_{\cV//\cW^\perp} \ U\cdot \cF^\#
$$ 
and hence, in general we get that
\beq\label{el S vdual}
 S_{(U\cdot \cF)^\#_\cV}=P_{\cV//\cW^\perp}\,  U\, S_{\cF^\#} \, U^* \,P_{\cV//\cW^\perp}^*\ .
 \eeq
We can now argue as in the proof of Theorem \ref{teo acot lids multi} considering $M= |P_{\cV//\cW^\perp}|$ (and 
the polar decomposition $P_{\cV//\cW^\perp}=V \, M$) and conclude that 
\beq\label{el s vdual3}
\la(S_{(U\cdot \cF)^\#_\cV})=
\big(\,  \la\left( M_\cW \, 
(\, U\, S_{\cF}^\dagger \, U^*)_\cW \, M_\cW \, 
\right)
\coma 0_{|\M|-d}\, \big)\ , 
\eeq 
where in general $S_\cW\in B(\cW)$ denotes the restriction of $S$ to its reducing subspace $\cW$.
Using Eq. \eqref{eq ang y vec}, if $U_0$ is as in item 1 then
$$ 
M_\cW \,
\ (\, U_0\, S_{\cF}^\dagger \, U_0^*)_\cW\ M_\cW \, \  w_j
= \cos^{-2}(\theta_j)\,\la_{d-j+1}^{-1}\ w_j \peso{ for }  j\in\IN{d}\ . 
$$
Thus, the previous facts together with Eq. \eqref{el s vdual3} show item 1.

\pausa
In case $U\in\op$ is a unitary operator such that $U(\cW)=\cW$ then Theorem \ref{teo acot lids multi} implies that 
\beq\label{ec se usa despues1}
\prod_{j\in\IN{k}}\la_j(S_{(U_0\,\cF)^\#_\cV})\leq \prod_{j\in\IN{k}}\la_j(S_{(U\,\cF)^\#_\cV}) 
\quad ,\quad k\in\IN{d}\ . 
\eeq
As explained in Section \ref{sec 2.2.}, Eq. \eqref{ec se usa despues1} implies that 
\beq\label{ec se usa despues2}
\sum_{j\in\IN{k}}\la_j(S_{(U_0\,\cF)^\#_\cV})\leq \sum_{j\in\IN{k}}\la_j(S_{(U\,\cF)^\#_\cV}) 
\quad ,\quad k\in\IN{d}\ . 
\eeq 
If $h\in\convf$ is non-decreasing then, by the submajorization relation in Eq. \eqref{ec se usa despues2}, we conclude that 
$$
P_h((U_0\cdot \cF)^\#_\cV)=\sum_{j\in\IN{d}}h(\la_j(S_{(U_0\,\cF)^\#_\cV}))\leq\sum_{j\in\IN{d}}h(\la_j(S_{(U\,\cF)^\#_\cV}))
=P_h((U\cdot \cF)^\#_\cV)\ , 
$$ which proves Eq. \eqref{eq desi pot h}. Similarly, the final claim follows from Eq. \eqref{el s vdual3} and Lidskii's multiplicative inequality, as stated in Theorem \ref{teo hay max y min mayo} and the properties of log-majorization described in Section \ref{sec 2.2.}.
\end{proof}

\pausa
Consider the Notations \ref{notas}. Then, the previous theorem describes the rigid rotations $U_0$ that leave invariant $\cW$ and such that spectral structure $\la((U_0\cdot\cF)^\#_\cV)$ of the oblique canonical $\cV$-dual of $U_0\cdot\cF$ is optimal with respect to log-majorization. 

\pausa
On the other hand, for a fixed rigid rotation $U$ that leaves invariant $\cW$ and for a fixed $t\geq \tr(S_{(U\cdot\cF)^\#_\cV})$ 
Theorem \ref{sobre nuduales optimos} describes the spectral  structure $\la^\#_{\cV,\,t}(U\cdot\cF)$
of those oblique $\cV$-duals $\cG^{\rm op}_t(U)\in \cD_t(U\cdot\cF)$ that simultaneously minimize every convex potential within the set $\cD_t(U\cdot\cF)$. It is then natural to wonder whether the spectral structure $\la^\#_{\cV,\,t}(U_0\cdot\cF)$ of 
$\cG^{\rm op}_t(U_0)$ (optimal dual with trace restriction based on an optimal rigid rotation of $\cF$) has some 
optimality property. In order to tackle this problem we consider the following result.

\begin{lem}\label{lem comp entre nut op} 
Let $\la=(\la_i)_{i\in\IN{d}} \coma \mu=(\mu_i)_{i\in\IN{d}}\in\R_{\geq 0}^d\,$, let $m\leq 0$ be an integer and assume that $\la \prec_{w} \mu$. If we let $\nulamt $ and $\numumt $ be as in Remark \ref{rem estruc espec Ut} (based on $\la$ and $\mu$ respectively) for some $t\geq \tr\mu\ (\geq \tr\la)$ then we have that $\nulamt \prec\numumt $.
\end{lem}

\begin{proof} Recall that by construction $\tr(\nulamt )=\tr(\numumt )=t$. Hence, in case $\nulamt =\frac{t}{d}\cdot\uno$ the result follows from the well known relation $ \frac{t}{d}\cdot\uno\prec \rho$ for every $\rho\in\R^d$ such that $\tr\rho=t$.
Otherwise (see Remark \ref{rem estruc espec Ut}), there exists $1\leq r\leq d-1$ such that 
$$ 
\nulamt =(\la_1 \coma \ldots \coma \la_{r} \coma c\cdot \uno_{d-r}) \peso{with} c\leq \la_r\ .$$
On the other hand we can write $\numumt  =(\al\coma  \beta)\in(\R_{\geq 0}^d)\da$ where 
$$
\al=(\mu_i + (c\,'-\mu_i)^+)_{i=1}^{r}\in (\R_{\geq 0}^r)\da \peso{and} \beta=(\mu_i + (c\,'-\mu_i)^+)_{i=r+1}^{d}\in (\R_{\geq 0}^{d-r})\da\ .$$
Therefore, for every $k\in\IN{r}$ we have that
$$\sum_{i\in \IN{k}} \la_i \leq \sum_{i\in \IN{k}} \mu_i\leq \sum_{i\in \IN{k}} (\mu_i + (c\,'-\mu_i)^+) \implies (\la_i)_{i\in \IN{r}}\prec_w \alpha \ ,$$
where we have used that $\la\prec_w\mu$ in the first inequality above. By \cite[Lemma 5.6]{MR12} we conclude that 
$ \nulamt \prec_w \numumt $. The result now follows from the equality $\tr(\nulamt )=\tr(\numumt )$.
\end{proof}

\begin{teo}\label{el mejor de los mas mejores}
Consider the Notations \ref{notas} and assume that $n\geq 2d$ (i.e. $2d-n\leq 0$).   
Let $\{x_i\}_{i\in \IN{d}}\in \cW^d$ be an ONB of eigenvectors for $S_\cF$ on $\cW$ i.e., 
such that $S_\cF \, x_j=\lambda_j\,x_j$ for every $ j\in \IN{d}\,$. 
Let $U_0\in\op$ be a unitary operator such that $U_0 \, x_j= w_{d-j+1}$ for every $ j\in \IN{d}\,$. 
Then,
\ben
\item If $U\in\op$ is a unitary operator such that $U(\cW)=\cW$ then $\tr(S_{(U_0\cdot\cF)^\#_\cV})\leq \tr(S_{(U\cdot\cF)^\#_\cV})$.
\item If $t\geq \tr(S_{(U\cdot\cF)^\#_\cV})$ and we let $\cG^{\rm op}_t(U)\in\cD_{\cV,\,t}(U\cdot\cF)$ (resp.
$\cG^{\rm op}_t(U_0)\in\cD_{\cV,\,t}(U_0\cdot\cF)$) be the optimal dual as in Theorem \ref{sobre nuduales optimos} then
 for every $h\in\convf$
\beq\label{opti tutti}
 P_h( \cG^{\rm op}_t(U_0)) \leq P_h(\cG^{\rm op}_t(U))\ . 
 \eeq
\een
\end{teo}
\begin{proof}
As explained in the proof of Theorem \ref{teo rotop} if $U$ and $U_0$ are as above then Eq. \eqref{ec se usa despues2} holds. In this case $$\tr(S_{(U_0\cdot\cF)^\#_\cV})=
\sum_{j\in\IN{d}}\la_j(S_{(U_0\,\cF)^\#_\cV})
\leq \sum_{j\in\IN{d}}\la_j(S_{(U\,\cF)^\#_\cV})=
\tr(S_{(U\cdot\cF)^\#_\cV}) \ ,$$
which shows item 1. On the other hand, if $t\geq \tr(S_{(U\cdot\cF)^\#_\cV})\geq \tr(S_{(U_0\cdot\cF)^\#_\cV})$ then Eq. \eqref{ec se usa despues2} together with 
Lemma \ref{lem comp entre nut op} and Theorem \ref{sobre nuduales optimos} (notice that in this case $m=2d-n\leq 0$)
imply that 
$$
\sum_{j\in\IN{k}}\la_j(S_{\cG^{\rm op}_t(U_0)})\leq \sum_{j\in\IN{k}}\la_j(S_{\cG^{\rm op}_t(U)})\quad , \quad k\in\IN{d}\ .$$
 Hence, Eq. \eqref{opti tutti} follow from the properties of majorization described in Section \ref{sec 2.2.} and Definition \ref{pot generales}.
\end{proof}

\pausa
We conjecture that Theorem \ref{el mejor de los mas mejores} is also true in case $2d-n \in\IN{d-1}\,$. We shall consider this problem elsewhere.

\section{Aliasing in oblique duality}\label{sec 5}

Let $\cW,\, \cV\subset \hil$ be closed subspaces such that $\cV\oplus \cW^\perp=\hil$ (or equivalently $\cW\oplus \cV^\perp=\hil$). Recall that in this context the aliasing norm associated to the consistent sampling
\beq\label{overall consist}
f\mapsto \tilde f=P_{\cW// \cV^\perp} f \ \ , \ \ f\in \hil 
\eeq
(see \cite{YEldar1,Janssen}) is 
given by
\beq\label{alias} 
A(\cW,\cV)=\sup_{e\in \cW^\perp\setminus \{0\}} \frac{\|P_{\cW// \cV^\perp}\, e\|}{\|e\|} 
= \|P_{\cW// \cV^\perp} \, P_{\cW^\perp}\|\ .
\eeq
Notice that the aliasing norm measures the incidence of the orthogonal complement of $\cW$ in the overall (oblique) encoding-decoding scheme in Eq. \eqref{overall consist} based on these two subspaces. We can interpret $A(\cW,\cV)$ as a measure of the amount of noise that we would get in the oblique encoding-decoding scheme when sampling a perturbed signal $f+e$ that has a component $e\in \cW^\perp$. This phenomenon is of interest only when $\cV\neq \cW$ (as $A(\cW,\cW)=0$).

\begin{lem}\label{norma prod}
Consider the Notations \ref{notas}. Then
\ben
\item $|P_{\cW\orto} P_\cV| \ v_i= \sin (\theta_i)\, v_i $ for every $i\in\IN{d}\,$.
\item $|P_{\cW\orto} P_{\cV//\cW\orto}|^2 \,w_i
=\tan^2(\theta_i) \,w_i$ for every $i\in\IN{d}\,$.
\een
\end{lem}
\proof
Recall that with the Notations \ref{notas}, $\{v_i\}_{i\in\IN{d}}$ is an ONB of $\cV$ such that
$|P_\cW\,P_\cV| \ v_i= \cos(\theta_i)\, v_i\,$, for $i\in\IN{d}\,$.
In this case, 
\begin{equation}\label{ident util1}
|P_{\cW\orto} P_\cV|^2=P_\cV- |P_\cW\, P_\cV|^2
\ \ \implies \ \ |P_{\cW\orto} P_\cV| \ v_i= \sin (\theta_i)\, v_i
\peso{for every} i\in\IN{d}\ .
\end{equation}
To prove item 2, let us fix $i\in\IN{d}\,$. 
By Eq. \eqref{sin primas} we know that 
$P_\cW P_\cV \, v_i = P_\cW \,v_i=\cos(\theta_i)\, w_i\,$.
On the other hand, recall from  Eq. \eqref{ec proy ob} that  
\beq\label{viwi}
(P_{\cV//\cW^\perp}) ^\dagger=P_\cW\,P_\cV \implies 
P_{\cV//\cW\orto}\, w_i 
= \cos(\theta_i)^{-1}\,v_i
\ ,
\eeq
since $v_i \in \cV = (\ker \, P_\cW\,P_\cV)\orto $ and 
$w_i \in \cW = R(P_\cW\,P_\cV)$. 
Similarly, we get that 
\beq\label{viwisi}
P_{\cW//\cV\orto} \,v_i \stackrel{\eqref{viwi}}{=}
P_{\cV//\cW\orto}^*( \cos(\theta_i)P_{\cV//\cW\orto} w_i  )=
 \cos(\theta_i) \,|P_{\cV//\cW\orto}|^2 \, w_i 
\stackrel{\eqref{eq ang y vec}}{=} \cos(\theta_i)^{-1}\,w_i\,.
\eeq
On the other hand, 
$$ 
|P_{\cW\orto}P_{\cV//\cW\orto}|^2=|P_{\cW\orto}\ P_\cV \ P_{\cV//\cW\orto}|^2
= P_{\cW//\cV\orto} \ |P_{\cW\orto}P_\cV |^2\ P_{\cV//\cW\orto}\ .
$$ 
Hence, using the previous remarks and Eq. \eqref{ident util1} we get that, 
for every $i\in\IN{d}\,$, 
$$
\barr{rcl}
|P_{\cW\orto} P_{\cV//\cW\orto}|^2 \,w_i
& \stackrel{\eqref{viwi}}{=} & 
P_{\cW//\cV\orto} \ |P_{\cW\orto}P_\cV |^2 (\cos(\theta_i)^{-1} \,v_i)
\\&&\\ 
&\stackrel{\eqref{ident util1}}{=}
&\frac{\sin(\theta_i)^2}{\cos(\theta_i)} \ P_{\cW//\cV\orto} \,v_i
\stackrel{\eqref{viwisi}}{=} 
\tan^2(\theta_i) \,w_i\ . 
\earr
$$
This completes the proof. \QED

\pausa
Consider now the Notations \ref{notas} , so in particular $\cV,\,\cW$ are finite dimensional. Then, using that $\|P_{\cW// \cV^\perp} \|=\cos(\theta_d)^{-1}$  we get $A(\cW,\cV)\leq \cos(\theta_d)^{-1}$ (i.e. with the notations of Remark \ref{rem comp angulos} we get $A(\cV,\cW)\leq \cos(\theta_{\cW,\,\cV})^{-1}$, see \cite{YEldar1}). Nevertheless, the previous bound for $A(\cW,\cV)$ is  not sharp: in case $\cV=\cW$ then $A(\cW,\cW)=0$, yet $\cos(\theta_d)^{-1}=1$. Next we compute the exact value of the aliasing norm:

\begin{cor}\label{pro mejor cota}
Consider the Notations \ref{notas}. Then, $$A(\cW,\cV) =   \tan (\theta_d)\ . $$
\end{cor}
\proof
By the definition of the aliasing \eqref{alias} and Lemma \ref{norma prod}  
\beq
A(\cW,\cV) = \| P_{\cW// \cV^\perp}\, P_{\cW\orto}\| = 
\| \,|P_{\cW\orto} P_{\cV//\cW\orto}| \,\|
=  \max _{i\in\IN{d}} \, \tan (\theta_i) = \tan (\theta_d) \ .\QEDP
\eeq

\pausa
Consider the Notations \ref{notas}. Let $\cG=\{g_i\}_{i\in \IN{d}}\in \cD_\cV(\cF)$ i.e.,
such that $T_\cF\,T_\cG^*=P_{\cW//\cV\orto}\,$.
 Then, when applying the encoding-decoding scheme induced by the
 pair $(\cF,\cG)$ the orthogonal complement $\cW\orto$ may also have an
 incidence in the sampling process; in case $\cV\neq \cW$, if we sample the
 perturbation $f+e$ for $f\in\cW$ and $e
 \in\cW\orto$ then, there is a corresponding perturbation of the coefficients
 $T_\cG^* f$ given by $T_\cG^*e$: in this case, the squared norm (energy) of the perturbation  is $\|T_\cG^* e
 \| ^2= \langle S_\cG\, e \coma e \rangle$. Thus, we
 introduce the following

 \begin{fed}\label{fed frame aliasing}
Let $\cW,\, \cV\subset \hil$ be closed subspaces
such that $\cW\oplus \cV^\perp=\hil$. Let $\cF=\{f_i\}_{i\in I}$
and $\cG=\{g_i\}_{i\in I}$  be frames for $\cW$ and $\cV$
respectively such that $\cG\in \cD_\cV(\cF)$. Then, we define the
aliasing relative to the oblique dual pair $(\cF,\cG)$, denoted
$A(\cF,\cG)$ given by 
\beq 
A(\cF,\cG)= \sup_{e\in
\cW\orto\setminus\{0\}} \frac{\|T_\cG^* \, e\|}{\|e\|}=
\sup_{e\in \cW\orto\setminus\{0\}} \frac{\langle \, S_\cG \, e\,
,\, e\,\rangle^{1/2}}{\|e\|} \ .\EOEP
\eeq
 \end{fed}

\pausa
 With the notations of Definition \ref{fed frame aliasing}, notice that $A(\cF,\cG)$ is a
normalized measure of the relative incidence of $\cW\orto$ in the
analysis of perturbed signals in terms of $\cG$, in the sense that
$$
\|T_\cG^* \,e\|\leq A(\cF,\cG) \ \|e\| \ \ , \ \ e\in \cW\orto \ .$$

\pausa
There is an alternative interpretation of the aliasing $A(\cF,\cG)$ that is as follows:
with the previous notations, let $e\in \cW\orto$: then
$$ 
\|T_\cG^* \,e\|^2= \sum_{i\in I} |\langle e\coma g_i\rangle |^2=
\sum_{i\in I} |\langle e\coma P_{\cW\orto} g_i\rangle |^2 
=\|T_{\widehat \cG}^* \, e \|^2 \ ,
$$ 
where $\widehat \cG= \{P_{\cW\orto} g_i\}_{i\in\In}$ is considered as a finite sequence in $\cW\orto$. Therefore, 
$$
A(\cF,\cG)=\sup_{e\in\cW\orto\setminus\{0\}}  
\left(\sum_{i\in I} \frac{|\langle e\coma P_{\cW\orto}g_i\rangle |^2}{\|e\|^2}\right)^{1/2} 
=\sup_{e\in\cW\orto\setminus\{0\}}   \frac{\|T_{\widehat\cG}^* \ e \|}{\|e\|}\  $$
can be considered as a measure of the (normalized) residual sampling power of $\widehat\cG=P_{\cW\orto}\cdot\cG$ in $\cW\orto$.

\pausa
Assume further that $\cF=\{f_i\}_{i\in I}$ is a Parseval frame for $\cW$ i.e. $S_{\cF}=P_\cW$. Then, $\cF^\#_\cV\in \cD_\cV(\cF)$ is such that $S_{\cF^\#_\cV}=|P_{\cW//\cV^\perp}|^2$; hence in this case 
$$
\cA(\cF,\cF^\#_\cV)=\sup_{e\in \cW^\perp\setminus\{0\}}\frac{\langle S_{\cF^\#_\cV} \,e\coma e\,\rangle^{1/2}}{\|e\|}
=\sup_{e\in \cW^\perp\setminus\{0\}}\frac{\|P_{\cW//\cV^\perp}e\|}{\|e\|}=\cA(\cW,\cV)\ . $$

\pausa
In opposition to $A(\cW,\cV)$, the aliasing $A(\cF,\cG)$ depends on
the particular choice of oblique dual frames $(\cF,\,\cG)$
considered, and not only on the subspaces $\cV$ and $\cW$. 
Then, it is natural to consider the problem of designing frames
$\cF$ and $\cG$ for $\cW$ and $\cV$ respectively, such that
$(\cF,\cG)$ is an oblique dual pair and such that $A(\cF,\cG)$ is
minimum.
\begin{rem}\label{lo mejorrr}
Consider the notations \ref{notas}. 
Given $\cG\in\cD_\cV(\cF)$, it is easy to see from Definition \ref{fed frame aliasing} that
\beq\label{alinorm}
 A(\cF,\cG) = \|T^*_{\cG}\,P_{\cW^\perp}\|=
 \|P_{\cW^\perp} \,S_{\cG}\,P_{\cW^\perp}\|\rai 
\eeq
On the other hand, as a consequence of Proposition \ref{pro tec} we see 
that $S_{\cF^\#_\cV}\leq S_\cG \implies 
P_{\cW^\perp} \,S_{\cF^\#_\cV}\,P_{\cW^\perp}\leq P_{\cW^\perp} \,S_{\cG}\,P_{\cW^\perp} \,$. Therefore 
$$ 
%P_{\cW^\perp} S_{\cF^\#_\cV}P_{\cW^\perp}\leq P_{\cW^\perp} S_{\cG}P_{\cW^\perp} \implies 
A(\cF,\cF^\#_\cV)=\|P_{\cW^\perp}  S_{\cF^\#_\cV}P_{\cW^\perp}\|\rai 
\leq  \|P_{\cW^\perp} S_{\cG}P_{\cW^\perp}\|\rai =A(\cF,\cG)\ .
$$
That is, in this case the canonical oblique dual pair $(\cF,\cF^\#_\cV)$ minimizes the aliasing among all 
oblique dual pairs $(\cF,\cG)$ for $\cG\in\cD_\cV(\cF)$. \EOE
\end{rem} 

\pausa
Consider the notations \ref{notas} and let $U\in B(\hil)$ be a unitary operator such that $U(\cW)=\cW$. As shown in previous sections, the spectral structure of $S_{(U\cdot \cF)^\#_\cV}$ depends on the choice of such a unitary. Therefore it is natural to consider the unitary operators $U$ as before, that minimize the aliasing $A(U\cdot \cF, \cG)$ where $\cG\in \cD_\cV(U\cdot \cF)$. Remark \ref{lo mejorrr} shows that in this case we can restrict attention to the oblique dual pairs of the form $(U\cdot \cF,(U\cdot \cF)^\#_\cV)$ for a unitary operator $U\in B(\hil)$ such that $U(\cW)=\cW$. 
 The following result fits into the previous analysis
scheme and links optimal solutions of this problem to optimal
solutions of the problem considered in Theorem \ref{teo rotop}.

\begin{teo}\label{teo chau chau}
Consider the Notations \ref{notas}. 
Let $\{x_i\}_{i\in \IN{d}}\in \cW^d$ be an ONB of eigenvectors for $S_\cF$ on $\cW$ i.e., 
such that $S_\cF \, x_j=\lambda_j\,x_j$ for every $ j\in \IN{d}\,$. 
\begin{enumerate}
\item If $U\in \op$ is a unitary operator such that $U(\cW)=\cW$ then
$$
A(U\cdot\cF \coma (U\cdot\cF)^\#_\cV)\geq 
\max_{j\in\IN{d}} \ 
\frac{\tan(\theta_j)}{\la_{d-j+1}\rai} 
\ .
$$
\item If $U_0\in\op$ is a unitary operator such that $U_0 \, x_j= w_{d-j+1}$ for every $j\in\IN{d}\,$,
then 
$$
A(U_0\cdot\cF \coma  (U_0\cdot\cF)^\#_\cV)=
\max_{j\in\IN{d}} \ 
\frac{\tan(\theta_j)}{\la_{d-j+1}\rai} 
\ .
$$
In particular, the lower bound in item 1. above is sharp.
\end{enumerate}
\end{teo}
\proof 
Let $U\in \op$ be a unitary operator such that $U(\cW)=\cW$. Then, as in the proof of Theorem \ref{teo rotop}, we have that 
$$ S_{(U\cdot \cF)^\#_\cV}= P_{\cV//\cW\orto} \, U \, S_{\cF^\#} \,U^*\, P_{\cW//\cV\orto} $$
Therefore, by Eq. \eqref{alinorm} we deduce that 
\begin{equation}\label{ident Ali}
A((U\cdot \cF)^\#_\cV\coma U\cdot \cF)^2
%=\sup_{f\in \cW\orto\setminus\{0\}} \frac{\langle \, S_{(U\cdot \cF)^\#_\cV}\, f\,,\, f\, \rangle }{\|f\|^2}
= \| P_{\cW\orto} P_{\cV//\cW\orto} \, U \, S_{\cF^\#} \,U^*\, P_{\cW//\cV\orto} P_{\cW\orto}\|\ .
\end{equation}
On the other hand, by Lemma \ref{norma prod}, we have that 
\beq\label{los wi autos}
|P_{\cW\orto} P_{\cV//\cW\orto}|^2 \,w_i
=\tan^2(\theta_i) \,w_i \peso{for every} i\in\IN{d}\ .
\eeq
In particular, $\rk(|P_{\cW\orto} P_{\cV//\cW\orto}|)
=\#(\{i\in \IN{d} \ : \ \theta_i>0\})\igdef k$. Let $P$ denote the 
the orthogonal projection onto $\cN = R(|P_{\cW\orto} P_{\cV//\cW\orto}|)\,$. Then, by the interlacing inequalities 
(see \cite{Bat}) we get that the eigenvalues $\lambda(P(US_{\cF^\#}U^*)P)=(\mu_i)_{i\in \M}$ satisfy  that: 
\begin{equation}\label{desig mus}
\mu_i \geq \lambda_{(d-k+i)}(S_{\cF^\#})=\la_{k-i+1}^{-1} \ \text{ if } \  1\leq i\leq k  \py 
\mu_i=0 \ \text{ if } \ i\geq k+1 \ . 
\end{equation}
We apply Lidskii's multiplicative inequality in Theorem 
\ref{teo hay max y min mayo} (to the operators acting on $\cN$, where $|P_{\cW\orto} P_{\cV//\cW\orto}|$ acts as an invertible operator) and get
$$
\barr{rcl}
\big(\,\tan^2(\theta_{d-k+i})\,\big)_{i\in\IN{k}}\circ (\mu_i)_{i\in\IN{k}} 
& \prec_w &
\left(\lambda_i(|P_{\cW\orto} \,P_{\cV//\cW\orto}| \, 
(P\, U \, S_{\cF^\#} \,U^*\, P)\, 
| P_{\cW\orto}\, P_{\cW//\cV\orto}|)\right)_{i\in\IN{k}}
\\&&\\ 
& = &\left(\lambda_i(P_{\cW\orto} P_{\cV//\cW\orto} \, U \, S_{\cF^\#} \,U^*\, 
P_{\cW//\cV\orto} P_{\cW\orto})\right)_{i\in\IN{k}}\in\R^k \ ,
\earr
$$ 
where the last equality follows by taking polar decomposition of 
$P_{\cW\orto} P_{\cV//\cW\orto}\,$. 
In particular, using the sub-majorization relation, Eq. \eqref{ident Ali} and the inequalities in \eqref{desig mus} we see that 
$$ 
\barr{rl}
A(U\cdot \cF,(U\cdot \cF)^\#_\cV )^2&\geq \max\limits_{i\in\IN{k}}\Big\{\frac{\tan^2(\theta_{d-(k-i)})}{\lambda_{k-i+1}} 
\Big\} =\max\limits_{i\in\IN{d}}\Big\{\frac{\tan^2(\theta_{j})}{\lambda_{d-j+1}} 
\Big\}\ , \earr
$$
which shows item 1. %In order to  prove item 2, fix 
Fix the unitary $U_0 \in \op$ of item 2. %such that 
Then $(U_0 \, S_{\cF^\#_\cV}\,U_0^*)\, w_i= \la_{d-i+1}^{-1} \,w_j$ for $i\in \IN{d}\,$. Recall 
from Eq. \eqref{los wi autos} that 
 $|P_{\cW\orto} P_{\cV//\cW\orto}|\,w_i
 =\tan(\theta_i) \, w_i$ for $i\in\IN{d}\,$.  Then 
$$
\la\big(\,|P_{\cW\orto}\, P_{\cV//\cW\orto} |
\, U_0 \, S_{\cF^\#} \,U_0^*\, |P_{\cW\orto}\, P_{\cW//\cV\orto}|\, \big)
=\left(\, \left(\,\frac{\tan^2(\theta_{j})}{\lambda_{d-j+1}}\,\right)_{j\in\IN{d}}\da \coma 0_{|\M|-d}\,\right) \ . 
$$
This shows that canonical oblique dual pair corresponding to $U_0 \cdot \cF$ attains the minimal aliasing. 
\QED

\begin{rem}
Consider the Notations \ref{notas}.  If $\cG\in\cD_\cV(\cF)$ then the compression 
$$
(S_\cG)_{\cW^\perp}=P_{\cW^\perp}S_\cG\,|_{_{\cW^\perp}} \in B(\cW^\perp)^+
$$ 
of $S_\cG$ to $\cW^\perp$ is a (operator valued) measure of the incidence of $\cW^\perp$ in the 
encoding-decoding scheme based on the oblique dual pair $(\cF,\cG)$. By Eq. \eqref{alinorm}, 
%Definition \ref{fed frame aliasing}, 
it follows that $A(\cF,\cG)=\|(S_\cG)_{\cW^\perp}\|\rai$, where $\|T\|$ stands for the operator norm of $T\in B(\cW^\perp)$. We can consider other (natural) scalar valued measures 
of the form 
$$
A_h(\cF,\cG)= \tr \, h((S_\cG)_{\cW^\perp}) \ , 
$$
for $h\in\convf$ non-decreasing and such that $h(0)=0$ (which is well defined since $(S_\cG)_{\cW^\perp}$ is a finite rank positive operator and $h(0)=0$). A careful inspection of the proof of Theorem \ref{teo chau chau} shows that, with the notations of that result, $$A_h(U_0\cdot \cF,(U_0\cdot \cF)_\cV^\#) \leq A_h(U\cdot \cF,(U\cdot \cF)_\cV^\#) $$ for every 
unitary operator $U\in B(\hil)$ such that $U(\cW)= \cW$, i.e. that the rigid rotation $U_0$ has several other optimal properties.
\EOE
\end{rem}

\pausa
{\bf Acknowledgments}. We would like to thank professors Gustavo Corach and Jorge Antezana for fruitful conversations related with the content of this work.

\end{document}